%% file: nestedpCenterArXiV.tex
\newif\ifArXiV

\ArXiVtrue	%ArXiV style
\ifArXiV
\documentclass{article}
\else
\documentclass[authoryear,preprint]{elsarticle}
\fi

\usepackage{fullpage}

\usepackage{lmodern}
\usepackage[english]{babel}

\usepackage{amsmath,amssymb,amsthm,graphicx}
\usepackage{comment}
\usepackage{enumitem}
\usepackage{todonotes}
\usepackage{multirow}
\usepackage{mathtools}
\usepackage{tikz}
\usepackage{hyperref}
\usepackage{subcaption}
\usepackage{booktabs}
\usepackage{xspace}
\usepackage{lscape}
\usepackage{multicol}
\usepackage{rotating} %for sidewaytable
\usepackage[normalem]{ulem} %for sout
\usepackage{placeins} % for \FloatBarrier = \clearpage without \newpage
\usepackage{graphicx}       % graphics
\usepackage{caption}
\usepackage{longtable}

\newtheorem{theorem}{Theorem}
\newtheorem{definition}[theorem]{Definition}

\newtheorem{proposition}[theorem]{Proposition}

\newtheorem{observation}[theorem]{Observation}
\newtheorem{example}{Example}

\theoremstyle{remark}

\usepackage[linesnumbered, noend, ruled]{algorithm2e}
\let\oldnl\nl% Store \nl in \oldnl
\newcommand{\nonl}{\renewcommand{\nl}{\let\nl\oldnl}}% Remove line number 
\newcommand{\nPC}{\hyperref[eq:nPC]{\texttt{(nPCA1)}}\xspace}
\newcommand{\nPCE}{\hyperref[eq:nPCE]{\texttt{(nPCA3)}}\xspace}
\newcommand{\nPCY}{\hyperref[eq:nPCY]{\texttt{(nPCA2)}}\xspace}
\newcommand{\nPCYR}{\hyperref[eq:nPCYR]{\texttt{(nPCR1)}}\xspace}
\newcommand{\nPCYRN}{\hyperref[eq:nPCYRN]{\texttt{(nPCR2)}}\xspace}
\newcommand{\tsplib}{\texttt{TSPlib}\xspace}
\newcommand{\pmed}{\texttt{pmed}\xspace}
\newcommand{\pCP}{$p$CP\xspace}
\newcommand{\npCP}{n-$p$CP\xspace}
\newcommand{\npCPA}{n-$p$CPA\xspace}
\newcommand{\npCPR}{n-$p$CPR\xspace}
\newcommand{\inI}{i \in \mathcal I}
\newcommand{\inJ}{j \in \mathcal J}
\newcommand{\inH}{h \in \mathcal H}
\newcommand{\noPP}{\texttt{B}\xspace}
\newcommand{\PP}{\texttt{P}\xspace}
\newcommand{\sH}{\texttt{PH}\xspace}
\newcommand{\R}{\texttt{(R)}\xspace}
\newcommand{\RN}{\texttt{(RN)}\xspace}
\newcommand{\RNL}{\texttt{(RNL)}\xspace}

\ifArXiV
\usepackage[affil-it]{authblk}
\usepackage[authoryear]{natbib}
\newenvironment{frontmatter}{}{}
\newenvironment{keyword}{\small \textbf{Keywords:}}{}
\let\address\affil

\fi

\begin{document}
		\begin{frontmatter}
			\title{Mixed-integer linear programming approaches for nested $p$-center problems with absolute and relative regret objectives}
			
			\ifArXiV
			\author[1,2,3]{Christof Brandstetter\thanks{christof.brandstetter@jku.at}}
			\author[2,3]{Markus Sinnl\thanks{markus.sinnl@jku.at}}
			\affil[1]{Institute of Production and Logistics Management, 
			Johannes Kepler University Linz, Linz, Austria}
            \affil[2]{Institute of Business Analytics and Technology Transformation, Johannes Kepler University Linz, Linz, Austria}
            \affil[3]{JKU Business School, Johannes Kepler University 
			Linz, Linz, Austria}
			\date{}
			\maketitle
			
			\else
			
			\author[jku,jkubus]{Brandstetter Christof}
			\ead{Brandstetter.Christof@jku.at}
			\author[jkubatt,jkubus]{Markus Sinnl}
			\ead{markus.sinnl@jku.at}
			\address[jku]{Institute of Production and Logistics Management, 
			Johannes Kepler University Linz, Linz, Austria}
			\address[jkubus]{JKU Business School, Johannes Kepler University 
				Linz, Linz, Austria}	
			\address[jkubatt]{Institute of Business Analytics and Technology Transformation, Johannes Kepler University Linz, Linz, Austria}
			\fi
%\maketitle

\begin{abstract}
We introduce the nested $p$-center problem, which is a multi-period variant of the well-known $p$-center problem. The use of the \emph{nesting} concept allows to obtain solutions, which are consistent over the considered time horizon, i.e., facilities which are opened in a given time period stay open for subsequent time periods. This is important in real-life applications, as closing (and potential later re-opening) of facilities between time periods can be undesirable.

We consider two different versions of our problem, with the difference being the objective function. The first version considers the sum of the absolute regrets (of nesting) over all time periods, and the second version considers minimizing the maximum relative regret over the time periods.

We present three mixed-integer programming formulations for the version with absolute regret objective and two formulations for the version with relative regret objective. For all the formulations, we present valid inequalities. Based on the formulations and the valid inequalities, we develop branch-and-bound/branch-and-cut solution algorithms. These algorithms include a preprocessing procedure that exploits the nesting property and also begins heuristics and primal heuristics.

We conducted a computational study on instances from the literature for the $p$-center problem, which we adapted to our problems. We also analyse the effect of nesting on the solution cost and the number of open facilities.

\end{abstract}
\end{frontmatter}

%\tableofcontents

\section{Introduction}
A crucial factor in long-term planning, particularly in \emph{multi-period facility location}, is consistency over the given \emph{time horizon} for planning. There exist many such facility location problems (see, e.g.,\cite{Laporte2019}), but they often offer solutions that are inconsistent with respect to the number of open facilities. This inconsistency can lead to the opening, closing, and potentially reopening of facilities over the given (discrete) time horizon. Such actions can incur significant monetary or environmental costs, which are undesirable. Moreover, as a consequence of this, such inconsistent solutions can be difficult to present to decision makers, as they usually are not compatible with the intuition of decision makers \cite{McGarvey2022}.
As early as 1971, the concept of modeling facility location problems to address such inconsistent solutions emerged \cite{Scott1971}, and in \cite{Roodman1975} proposed a \emph{nesting constraint} to deal with inconsistent solutions and introduced the first mixed-integer linear programming (MILP) formulation incorporating this constraint. 
%Such facility location problems fall within a subset of \emph{multi-period} problems (see Chapter 11 of \cite{Laporte2019} for a summary of multi-period location problems).
The nesting constraint has been utilized in some multiperiod facility location problems over the years (e.g., in \cite{albareda2009,bakker2024value,castro2017cutting,escudero2017solving}, see Section \ref{sec:literature} for more details); however, it was not explicitly referred to as "nesting" in these works, and the concept was not the focus of these works. The nesting concept, under this name, was revisited in \cite{McGarvey2022}, who applied it to the $p$-median problem. They suggested extending this nesting approach to other fundamental problems in location science, such as the \emph{maximum coverage problem}\footnote{During the literature review for this work we discovered that already in 1980 a nested maximum coverage problem was considered in \cite{Schilling1980}, for more details see Section \ref{sec:literature}. } or the \emph{(discrete) $p$-center problem}. In this work, we follow up on this suggestion by applying the nesting concept to the $p$-center problem to introduce the \emph{(discrete) nested $p$-center problem} (\npCP)\footnote{This work is an extension of the short paper \cite{Brandstetter2024} presented at INOC 2024 and the Master's thesis \cite{Brandstetter2024b} of the first author.} Note that there also exists a continuous version of the $p$-center problem (see, e.g., \cite{chen2009new}), however, in this work we focus on the discrete version of the problem.

Similarly as \cite{McGarvey2022} did for the $p$-median problem, we consider two versions of the \npCP. The input and the set of feasible solutions is the same for both versions, while the objective function differs. In the first version, the objective function consists of minimizing the sum of the absolute regret (of nesting) over the time horizon, and in the second version, the objective function consists of minimizing the maximum relative regret over the time horizon. 

The set of \emph{feasible solutions} for the \npCP is defined as follows:
\begin{definition}\label{def:npCP}
    Let $\mathcal I$ be a set of customer demand points, $\mathcal J$ be a set of potential facility locations with distances $d_{ij}\geq0$ between each $\inI$ and $\inJ$. Let
    $\mathcal{H}=\{1,\ldots, H\}$ denote the time horizon, and let $\mathcal P=\{p^1,\ldots p^H\}$ be a set of integers with $p^h\le p^{h+1}$ for $h=1,\ldots,H-1$ and $p^H\leq |\mathcal J|$.
    A \emph{feasible solution} $(\mathcal{J}^1, \dots \mathcal{J}^H)$ for the \npCP consists of a set $\mathcal J^h \subseteq \mathcal J$ with $|\mathcal J^h|=p^h$ for $h\in \mathcal H$. Moreover, the \emph{nesting constraint} must be fulfilled by these sets,
    i.e., $\mathcal J^h\subseteq \mathcal J^{h+1}$ must hold for $h=1,\ldots,H-1$.
\end{definition}

\begin{observation}
    In our definition of the \npCP\ we have that the number of facilities to be allowed open is non-decreasing over the time horizon. The optimal solution to this problem is also the optimal solution to the problem-variant, where the
    number of facilities to be allowed open is non-increasing over the time horizon (and the nesting constraint is accordingly adapted to $\mathcal J^{h+1}\subseteq \mathcal J^h$), as these problems are equivalent.
\end{observation}

Given a feasible solution and $h \in H$ we define the \emph{absolute regret} and \emph{relative regret} for this $h \in H$.
\begin{definition}
    For a given time period $\inH$ and set $\mathcal J^h$, let $d(\mathcal J^h)=\max_{i \in \mathcal \mathcal I} \min_{\inJ^h} d_{ij}$ and $d^{h*}$ be the optimal objective function value of the $p$-center problem for $p=p^h$. The \emph{absolute regret} for a given $\inH$ and set $\mathcal{J}^h$ is defined as $\mathcal{R_A}(h, \mathcal{J}^h) = d(\mathcal{J}^h)-d^{h*}$ and the \emph{relative regret} as $\mathcal {R_R}(h, \mathcal J^h) = \mathcal {R_A} (h, \mathcal J^h)/d^{h*}$. 
 \end{definition}

With these definitions we can define the two problems we consider in the paper, namely the \emph{(discrete) nested $p$-center problem with the sum of absolutes regrets objective} (\npCPA), and the \emph{(discrete) nested $p$-center problem with the minimizing the maximum relative regret objective} (\npCPR).

\begin{definition}
The \emph{(discrete) nested $p$-center problem with the sum of absolute regrets objective} (\npCPA) is defined as the problem of finding a feasible solution $(\mathcal{J}^1, \dots \mathcal{J}^H)$ to the \npCP
    which minimizes $\sum_{\inH}\mathcal{R_A}(h, \mathcal{J}^h)$, i.e., the sum of the absolute regret over all time periods.
The \emph{(discrete) nested $p$-center problem with the minimizing the maximum relative regret objective} (\npCPR) is defined as the problem of finding a feasible solution $(\mathcal{J}^1, \dots \mathcal{J}^H)$ to the \npCP which minimizes $\max_{\inH}\mathcal{R_R}(h, \mathcal{J}^h)$, i.e., the maximum relative regret over all time periods.    
\end{definition}

\begin{observation}\label{obs:ignore}
    The objective function of the \npCPA\ can be viewed as minimizing the sum of the distances $d(\mathcal{J}^h)$ over the time periods $\inH$, as the optimal objective function values $d^{h*}$ of the $p$-center problems for $p = p^h$,
    are constant for each time period $h$, they can be ignored in the optimization process.
\end{observation}

\begin{observation}
    For $|\mathcal H|=1$ the problems reduce to the (classical) $p$-center problem (\pCP) which was introduced by Hakimi \cite{Hakimi1964} in 1964. The \pCP\ is NP-hard for $p \geq 2$ \cite{kariv1979}.
\end{observation}

\begin{figure}[h!tb]
    \centering
    \includegraphics[width=0.7\linewidth]{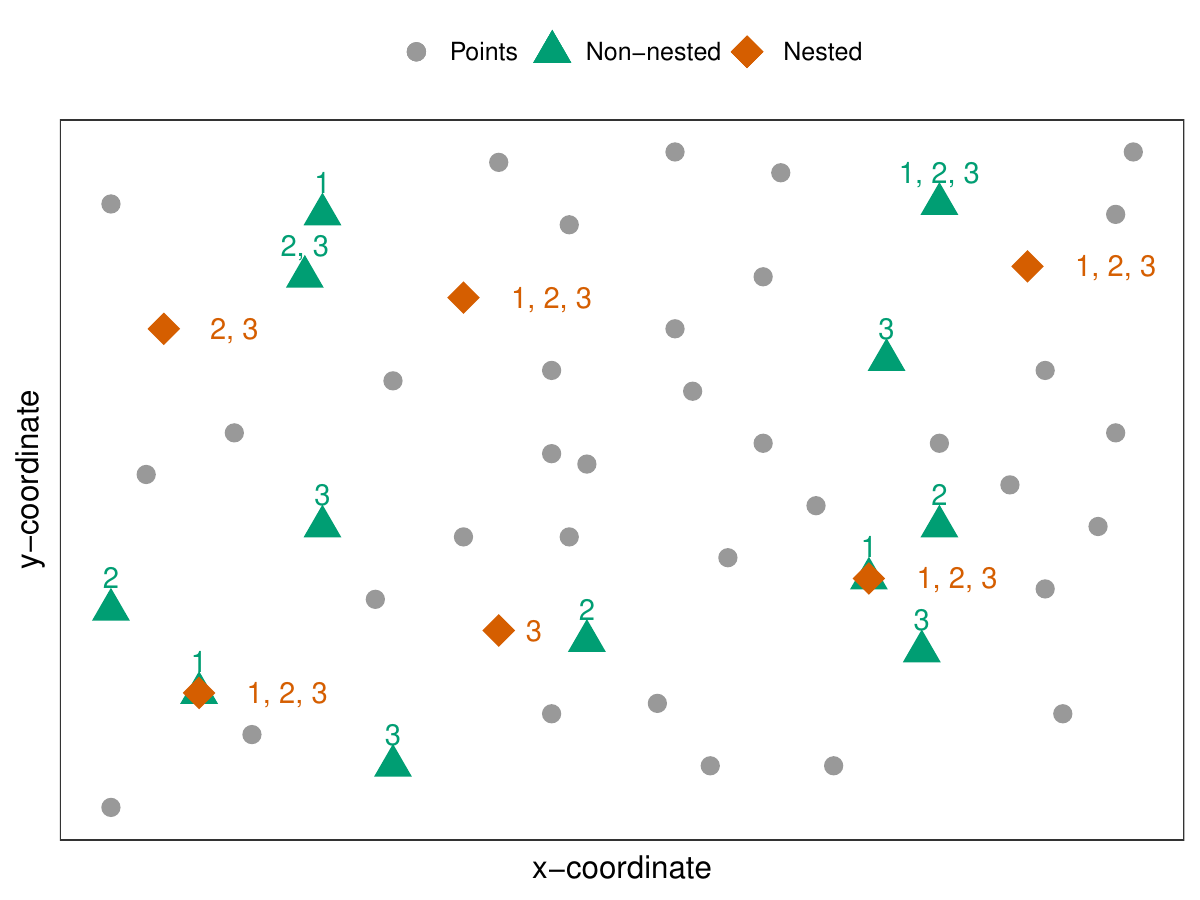}
    \caption{\textmd{Instance \texttt{eil51} with the optimal solutions for \npCPA (rectangles) with $\mathcal{P} = \{4, 5, 6\}$ and for \pCP (triangles) for $p=4,5,6$}}
    \label{fig:nest-non}
    %\vspace{-10mm}
\end{figure}

Figure \ref{fig:nest-non} shows an exemplary instance of the (nested) $p$-center problem (the \texttt{eil51} instance of the \tsplib\ \cite{Reinelt1991}). In this instance each point (visualized as grey dots) can be a potential facility
location or a customer demand point, so $\mathcal I=\mathcal J$ holds. The distance between two points in this instance is the Euclidean distance. In this figure, we illustrate an optimal solution of the \npCPA for $H = 3$ with
$\mathcal P = \left\{4,5,6\right\}$ indicated by the orange rectangles and the optimal solutions for the \pCP when individually solving it for $p=4,5,6$ indicated by the green triangles. The numbers beside or above indicate the
time period in which a facility was opened at this location. For example, if the numbers one, two, and three are above the green triangles, it means that in every individual solution of the $p$-center problem, a facility is opened at
this location. If the numbers two and three are depicted beside the orange rectangles, it means that at this location a facility was opened in the second time period and it could not be closed again as it is not allowed by the nesting.
The optimal solutions of the \pCP opened twelve different facilities over the three time periods and only one facility was open in all time periods, while the solution of the \npCPA only opened six facilities, so only half of the facilities
were opened. There might be optimal \pCP solutions that would require fewer facilities over the three time periods. For this instance, the sum of $d(\mathcal{J}^h)$ for $\inH$ of the \npCPA is 61 and the optimal objective function values of
the \pCP\ are 22, 19, 17 with the sum of 58. Thus, the absolute regret is 61-58=3.

\subsection{Literature review} \label{sec:literature}
There exists extensive work on the $p$-center problem, multi-period location problems in general and some work on nested facility location problems. There also exists a wide variety of different in-exact solution approaches like heuristics and approximation algorithms for the problems discussed in the following. We focus this review on exact solution approaches, as this work considers the development of exact solution algorithms.

\subsubsection{The \texorpdfstring{$p$}{p}-center problem} \label{sec:lit_pcenter}
The \pCP is a fundamental problem in location science with applications in countless areas such as emergency service location and planning of relief actions in humanitarian crisis \cite{calik2013double,jia2007modeling,lu2013robust} but also in clustering of large-scale data \cite{malkomes2015fast}, computer vision \cite{friedler2010approximation} and feature selection \cite{meinl2011maximum}. There also exist many variants of the \pCP, such as the $p$-next center problem \cite{albareda2015centers}, the $\alpha$-neighbor $p$-center problem \cite{gaar2023exact}, the capacitated $p$-center problem \cite{khuller2000capacitated} and countless others.

The \pCP was introduced by \cite{Hakimi1964} and the first solution algorithm for the \pCP\ was developed by \cite{minieka1970}. The author used the relationship of the \pCP to the \emph{set cover problem}, as solving the \pCP for a fixed distance can be transformed into set cover problem. This connection to the set cover problem was also used in more recent years by \cite{chen2009new,contardo2019scalable} which are amongst the state-of-the-art approaches for the \pCP. The classical MILP formulation can be found in textbooks like \cite[Chapter 5]{wiley2013}. However, the linear programming (LP) relaxation of this formulation is known to be quite bad and solution approaches based on it are thus not competitive with other exact solution approaches. There also exist other MILP formulations that have better LP-relaxation bounds \cite{Elloumi2018,calik2013double,Elloumi2004}. In \cite{GAAR2022} the authors present a branch-and-cut algorithm which is based on a projection of the classic formulation, which is also improved with valid inequalities such that the LP-relaxation bounds are similar to the bounds of \cite{Elloumi2018,calik2013double,Elloumi2004}. Our formulations for the \npCPA and the \npCPR are based on these formulations for the \pCP. For heuristic approaches to the \pCP we refer to the recent survey by \cite{garcia-diaz2019}. 

\subsubsection{Nested and other multi-period facility location problems}\label{sec:lit_nested}

Location problems needing to fulfill a nesting constraint can be categorized as one problem family within the area of the multi-period location problems (for a general overview of the area, see, e.g., \cite[Chapter 11]{Laporte2019}). To the best of our knowledge, the first work to discuss location problems, in which facilities are opened iteratively over a certain time horizon was \cite{Scott1971} in 1971. The authors compared a dynamic programming system that takes into account the complete
time horizon with a myopic system, which optimizes the next period without considering any later periods. Their dynamic programming system outperformed the myopic system for larger time horizons.
A first MILP formulation for such
problems was introduced in \cite{Roodman1975}, where the goal was to minimize the operational cost of closing facilities iteratively over a certain time horizon. This was also the first to use a nesting-type constraint which
enforces open facilities to be open until the end of the time horizon. In addition, the authors presented a generalization of the formulation, in which they start with a set of facilities that are open at the beginning and a set of potential
facilities that can be opened. The starting facilities can be closed at any point in the time horizon, but once closed, they must remain closed, while the potential facilities can be opened but not closed again. This allows for a restricted redistribution of facilities. 

Furthermore in 1980, a multi-objective and multi-period maximum coverage problem was presented by \cite{Schilling1980}, where the author developed a modeling framework that allows to explore the trade-off of planning strategies with a focus on present service goals and future service goals, by combining the goals in a multi-objective programming model. The model excluded the option of closing facilities by the nesting constraint. 
In 2009 Albareda-Sambola et al. proposed a MILP formulation for the incremental service facility location problem, where a certain service level has to be maintained in every time period, where service level means that a certain number of customers needs to be solved in a given time period. Moreover, as soon as a customer gets served in a time period, she also needs to be served in all subsequent time periods, and she needs to be served from the same facility in all time periods. This is implicitly causing a nesting constraint on the facilities in the solution. They solved their problem by means of a Langrangian dual approach combined with an ad hoc heuristic. Further incremental service facility location problems can be found in e.g. \cite{Ljubic2019, Kalinowski2015}. 
In 2017 Castro, Nasini and Saldanha-da-Gama presented a Benders decomposition approach for a capacitated multi-period facility location problem with a nesting constraint, see \cite{castro2017cutting}. 
Also in 2017, Escudero and Pizarro Romero introduced a quadratic 0-1 programming approach for a multi-period facility location problem where the solution has to fulfill nesting constraints. They developed a matheuristic algorithm to solve large scale instances, see \cite{escudero2017solving}. 
More recently, the nesting concept was revisited and applied to the $p$-median problem for two objective functions, minimizing the sum of the absolute regrets and minimizing the maximum relative regret by \cite{McGarvey2022}. The authors propose a heuristic based on Lagrangian relaxation to solve both problem-variants. 
Bakker and Nickel focus in a recent study on the value of the multi-period solution, which is a quantifier of the benefit resulting from using a multi-period model over a static one period model. To link the time periods together they used the nesting constraint, see \cite{bakker2024value}.

Aside from explicitly stating a nesting constraint, a further approach to linking the time periods is by penalizing the closing of facilities with high costs, see, e.g., \cite{galvao1992, Wesolowsky1975}. Recent advancements in multi-period location problems have increasingly focused on dynamic facilities, allowing for the relocation of these facilities over the planning horizon. For instance, Calogiuri et al. developed a heuristic to address the multi-period $p$-center problem, incorporating varying arc-traversal times and mobile facilities that can be relocated multiple times during the planning horizon. Correia and Melo extended the classical multi-period facility location problem by considering two customer segments, time-sensitive and time-insensitive customers. They start with an already existing network of facilities and determine new facilities and their capacities , but also relocate existing ones. In 2016 \cite{CORREIA2016} they proposed two MILP formulations to redesign the facility network at minimal cost, while in 2017 they further extended the problem by allowing for capacity changes over time \cite{Correia2017}. In 2020 Sauvey et al. developed a heuristic approach to this problem, see \cite{SAUVEY2020}. In \cite{GUDEN2019} developed a branch-and-price algorithm for the dynamic $p$-median problem with mobile facilities, where some facilities cannot be moved or relocated while mobile facilities can be relocated in each time period. Furthermore, in 2015 \cite{VANDENBERG2015} introduced a time-dependent probabilistic location model for emergency medical vehicles with the goal of maximizing expected coverage throughout the day while minimizing the number of open facilities and relocations.

\section{Mixed-Integer linear programming formulations for the \texorpdfstring{nested $p$-center problem}{nested p-center problem} with the sum of absolute regrets objective}\label{sec:milp-formulations}

In this section, we first present a formulation for the \npCPA based on the classical textbook formulation of the \pCP followed by a formulation based on the formulation of \cite{GAAR2022}. We present valid inequalities for both formulations. At the end of the section, we present a third formulation based on the formulation(s) of \cite{Elloumi2018, Elloumi2004}. For this third formulation, we present a variable-fixing procedure. In all three formulations we use Observation \ref{obs:ignore}, i.e., the fact that in the sum of absolute regrets objective the objective function values of the \pCP problems over the time horizon is just a constant which can thus be ignored. Note that for all three formulations, in case $|\mathcal H|=1$ the respective formulations of the \pCP are obtained.

\subsection{First formulation}\label{sec:xy}

The first formulation, denoted as \nPC, uses two sets of binary variables. The variable $x_{ij}^h$ indicates if customer demand point $\inI$ is assigned to potential facility location $\inJ$ in time period $\inH$ and the variable $y_j^h$ indicates if a facility is opened at the potential facility location $\inJ$ in time period $\inH$. The continuous variables $z^h$ measure the maximum distance from any customer demand point to its nearest open facility in time period $\inH$. Using these variables, the \npCPA can be formulated as follows.

\begin{subequations}
    \begin{align}
         & \nPC\  & \min         &  & \sum_{\inH}z^{h}          &                                      & \label{eq:nPC-objective}                                              \\
         &        & \text{ s.t.} &  & \sum_{\inJ}y_{j}^h        & =     p^{h}                          & \forall \inH                                  \label{eq:nPC-numOfFac} \\
         &        &              &  & \sum_{\inJ}x_{ij}^h       & =     1                              & \forall \inI, \inH                            \label{eq:nPC-oneToOne} \\
         &        &              &  & x_{ij}^h                  & \leq  y_j^{h}                        & \forall \inI, \inJ, \inH                      \label{eq:nPC-onlyOpen} \\
         &        &              &  & \sum_{\inJ}d_{ij}x_{ij}^h & \leq  z^{h}                          & \forall \inI, \inH                            \label{eq:nPC-zPush}    \\
         &        &              &  & y_{j}^h                   & \geq  y_j^{h-1}                      & \forall \inH \setminus \left \{1 \right \}    \label{eq:nPC-nested}   \\
         &        &              &  & x_{ij}^h,                 & \in   \left \{\text{0, 1}\right \}\  & \forall \inI, \inJ, \inH                      \label{eq:nPC-binX}     \\
         &        &              &  & y_{j}^h,                  & \in   \left \{\text{0, 1}\right \}   & \forall \inJ, \inH                            \label{eq:nPC-binY}     \\
         &        &              &  & z^{h}                     & \in   \mathbb R_{\ge 0}              & \forall \inH                                  \label{eq:nPC-nonnegZ}
    \end{align}  \label{eq:nPC}
\end{subequations}

The objective function \eqref{eq:nPC-objective} minimizes the sum over the distances $z^h$ over all time periods. The constraints \eqref{eq:nPC-numOfFac} ensure that $p^h$ facilities are opened in time period $h$.
Constraints \eqref{eq:nPC-oneToOne} ensure that each customer is assigned to only one facility in each time period. The constraints \eqref{eq:nPC-zPush} are pushing the decision variables $z^h$ to the maximum distance of any assigned
customer-facility combination in each time period. Each customer can only be assigned to an open facility, which is ensured by constraints \eqref{eq:nPC-onlyOpen}. The nesting constraints \eqref{eq:nPC-nested} ensure that each facility,
which is opened in time period $h$, is also open in time period $h+1$, so once a facility is opened in a time period, it cannot be closed in later time periods. Without this constraint, the formulation would just represent the sum over
the individual \pCP for each time period. The remaining constraints \eqref{eq:nPC-binX} - \eqref{eq:nPC-nonnegZ} are the binary constraints for the variables $x_{ij}^h$ and $y_{j}^h$ and the non-negativity constraint for
variables $z^h$. This formulation consists of $\mathcal{O}\left( \left| \mathcal{I} \right| \left| \mathcal{J} \right| \left| \mathcal{H} \right|\right)$ variables and constraints.

Given any lower bound on a decision variable $z^h$ (e.g., a lower bound on the objective value of the \pCP for period $h$), the following valid inequalities can be derived. They are based on Lemma 5 of \cite{GAAR2022} which proposed a similar idea for the \pCP. 
 
\begin{proposition}\label{prop:xy-lift}
    For a given $\inH$, let $LB^h \geq 0$ be a lower bound on the value of decision variable $z^h$ of \nPC for any optimal solution. Then
    \begin{equation}
        z^h \geq \sum_{\inJ} \max \left\{LB^h, d_{ij}\right\}x_{ij}^h \label{eq:nPC-OPT}
    \end{equation}
    is a valid inequality for \nPC\ i.e. every feasible solution of \nPC\ fulfills \eqref{eq:nPC-OPT}.
\end{proposition}
\begin{proof}
Suppose there exists a feasible solution $(\bar x,\bar y,\bar z)$ of \nPC which violates \eqref{eq:nPC-OPT} for a given $h \in H$. Thus, we must have $\bar z^h < \sum_{\inJ} \max \left\{LB^h, d_{ij}\right\}\bar x_{ij}^h$ for at least one customer $i$. Due to constraints \eqref{eq:nPC-oneToOne} and \eqref{eq:nPC-binX} exactly one variable $\bar x_{ij}$ for this customer $i$ has value one and all the others are zero, let $j'$ be the index of the variable which has value one. We proceed by a case distinction.

\begin{itemize}
\item $d_{ij'}>LB^h$: In this case the lifting does not have any effect. We obtain a contradiction to our assumption as due to constraints \eqref{eq:nPC-zPush} of \nPC we must have $\bar z^h\geq d_{ij'}$.
\item $d_{ij'} \leq LB^h$: The constraint \eqref{eq:nPC-OPT} for the given solution and the considered $i$ and $h$ reduces to $\bar z^h \geq LB^h$ which is valid by definition of $LB^h$. Thus we also arrive at contradiction. $\Box$
\end{itemize}
\end{proof}

Note that we need to have \emph{any} optimal solution in above proposition, as there can be multiple optimal solutions for an instance of the \npCPA and these solutions may have different values for the same $z^h$. For the other formulations, we present valid inequalities based on the same idea in the following sections, the proofs for them follow similar arguments as the proof above. In Section \ref{sec:algorithm} we detail how to use the valid inequalities in our solution algorithm.
%One value fulfilling this condition is for example the optimal solution of the $\pCP$ for $p=p^h$.

\subsection{Second formulation}\label{sec:y}
This formulation, denoted as \nPCY, uses the same variables as \nPC, except for the $x$-variables, which are not 
necessary for this formulation.
\begin{subequations}\label{eq:nPCY}
    \begin{align}
         & \nPCY & \min         &  & \sum_{\inH} z^{h}   &                                                                     & \label{eq:nPCY-objective}                                                  \\
         &       & \text{ s.t.} &  & \sum_{\inJ} y_{j}^h & = p^h                                                               & \forall \inH                                      \label{eq:nPCY-numOfFac} \\
         &       &              &  & z^{h}               & \geq d_{ij} - \sum_{j':d_{ij'} < d_{ij}} (d_{ij} - d_{ij'})y_{j'}^h & \forall \inI, \inJ, \inH                          \label{eq:nPCY-zPush}    \\
         &       &              &  & y_j^{h}             & \geq y_j^{h-1}                                                      & \forall \inJ, \inH \setminus \left \{1 \right \}  \label{eq:nPCY-nested}   \\
         &       &              &  & y_{j}^h             & \in \{\text{0, 1}\}                                                 & \forall \inJ, \inH                                \label{eq:nPCY-binY}     \\
         &       &              &  & z^{h}               & \in \mathbb R_{\ge 0}                                               & \forall \inH                                      \label{eq:nPCY-nonnegZ}
    \end{align}
\end{subequations}

The objective function is given as \eqref{eq:nPCY-objective} which minimizes the sum over the decision variables $z^h$ over all time periods. The constraints \eqref{eq:nPCY-numOfFac} ensure that only $p^h$ facilities are opened in
time period $h$. Constraints \eqref{eq:nPCY-zPush} push the decision variables $z^h$ to the largest distance of any customer to its nearest open facility. The nesting constraint is again given as \eqref{eq:nPCY-nested} and the
remaining constraints \eqref{eq:nPCY-binY} and \eqref{eq:nPCY-nonnegZ} are the binary and non-negativity constraints, respectively. This formulation consists of
$\mathcal{O}\left( \left| \mathcal{I} \right| \left| \mathcal{J} \right| \left| \mathcal{H} \right| \right)$ constraints and $\mathcal{O}\left( \left| \mathcal{J} \right| \left| \mathcal{H} \right| \right)$ variables.

%We can pose valid inequalities in a similar way as done for \nPC.
\begin{proposition}\label{prop:y-lift}
   For a given $\inH$, let $LB^h \geq 0$ be a lower bound on the value of decision variable $z^h$ of \nPCY for any optimal solution. Then
    \begin{equation}
        z^h \geq \max \left\{LB^h, d_{ij}\right\} - \sum_{j':d_{ij'} < d_{ij}} \left(\max \left\{LB^h, d_{ij}\right\} - \max \left\{LB^h, d_{ij'}\right\}\right)y_{j}^h \label{eq:nPCY-OPT} %\tag{nL-OPT}
    \end{equation}
    is a valid inequality for \nPCY\ i.e. every feasible solution of \nPCY\ fulfills \eqref{eq:nPCY-OPT}.
\end{proposition}
\begin{proof}\label{proof:y-lift}
    Similar to the proof of Proposition \ref{prop:xy-lift}. $\Box$
\end{proof}

\subsection{Third formulation}\label{sec:u}
The third formulation, denoted as \nPCE, uses the binary variable $y_j^h$ for $\inJ$ and $\inH$ to indicate the open facilities, analogously to the two previous formulations. Furthermore, let $\mathcal D = \{d_{ij}:\inI, \inJ\}$ denote the set of all possible
distances and let $D_1 \leq \ldots \leq D_K$ be the values contained in $\mathcal D$, so $\mathcal D = \{D_1, \dots, D_K\}$. Let $\mathcal K$ be the set of indices in $\mathcal D$. 
For a $k \in \mathcal K$ and $\inH$, the binary variable $u_k^h$ indicates if the objective function value in time period $h$ (measured by continuous variable $z^h$) is greater or equal to $D_k$. For customer $\inI$ let the set $S_i$ be the set
of indices $k \in \mathcal K$ for which there exists a facility $\inJ$ with $d_{ij}=D_k$. With this notation and variables, the \npCPA can be formulated as follows.

\begin{subequations}\label{eq:nPCE}
    \begin{alignat}{4}
         & \nPCE\  & \min        &  & \sum_{h \in \mathcal H} z^h                 &                   & \label{eq:nPCE-objective}                                                                        \\
         &         & \text{s.t.} &  & \sum_{j \in \mathcal J} y_{j}^h             & = p^h             & \forall \inH                                                            \label{eq:nPCE-numOfFac} \\
         &         & D_0 +       &  & \sum_{k=1}^{K}(D_k-D_{k-1})u^{h}_k & \leq z^h          & \forall \inH                                                            \label{eq:nPCE-zPush}    \\
         &         &             &  & u^{h}_k+\sum_{j:d_{ij}<D_k}y_{j}^h & \geq 1            & \forall i \in \mathcal{I}, \forall \inH, \forall k \in S_i \cup \{K\}   \label{eq:nPCE-uPush}    \\
         &         &             &  & u^{h}_k                            & \geq u_{k+1}^h    & \forall \inH, \forall k \in \mathcal{K} \setminus \{K\}                 \label{eq:nPCE-uNest}    \\
         &         &             &  & y_{j}^h                            & \leq y_{j}^{h-1}  & \forall \inH, \forall j \in \mathcal{J}                                 \label{eq:nPCE-nested}   \\
         &         &             &  & y_{j}^h                            & \in \{0, 1\}\quad & \forall \inH, \forall j \in \mathcal{J}                                 \label{eq:nPCE-biny}     \\
         &         &             &  & u^{h}_k                            & \in \{0, 1\}      & \forall h \in \mathcal{H}, \forall k \in \mathcal{K}                    \label{eq:nPCE-binU}     \\
         &         &             &  & z^h                                & \in \mathbb{R}    & \forall \inH                                                            \label{eq:nPCE-nonnegZ}
    \end{alignat}
\end{subequations}

The objective function \eqref{eq:nPCE-objective} minimizes the sum over the distances $z^h$ over all time periods. The correct value of the $z^h$-variables is ensured by constraints \eqref{eq:nPCE-zPush}. The constraints
\eqref{eq:nPCE-numOfFac} ensure that no more than $p^h$ facilities are opened in each time period. Constraint \eqref{eq:nPCE-uPush} is ensuring that if for any customer $i$ in time period $h$ no facility $j$ with smaller distance
than $D_k$ is opened $u_k^h$ has to be one. Since \eqref{eq:nPCE-uPush} is not defined for all $k \in \mathcal{K}$ but only for the subsets based on $S_i \cup \{K\}$, constraints \eqref{eq:nPCE-uNest} are necessary to ensure
that no $u_k^h$ can equal zero if $u_{k+1}^h$ is one (otherwise constraints \eqref{eq:nPCE-zPush} would not measure the distance correctly). The inequalities \eqref{eq:nPCE-nested} are for nesting and are the same as
(\ref{eq:nPC-nested},~\ref{eq:nPCY-nested}). The remaining constraints are the binary and non-negativity constraints, respectively. This formulation has
$\mathcal{O} \left( \left( \left| \mathcal{I} \right| + \left| \mathcal{K} \right| \right) \left| \mathcal{H} \right| \right)$ variables and
$\mathcal{O}\left( \min \left( \left| \mathcal{I} \right| \left| \mathcal J \right|, \left| \mathcal{I} \right| \left| \mathcal K \right| \right) \left| \mathcal{H} \right| \right)$ constraints.

Given upper bounds $UB^h$ and lower bounds $LB^h$ on $z^h$, some variables in the formulation can potentially be fixed to zero or one, and as a result of this fixing, some constraints can potentially be removed. 

\begin{proposition}\label{prop:npce-fixed}
    Let $UB^h$ be an upper bound on the value the decision variable $z^h$ can take in any optimal solution. The variables $u_k^h$ for every $k$ with $D_k > UB^h$ can then be set to zero (or removed from the formulation).
   Let $LB^h$ be a lower bound on the value of decision variable $z^h$ in any optimal solution. Then the variables $u_k^h$ for every $k$ with $D_k < LB^h$ can be set to one. As a consequence of these variable fixings, the constraints \ref{eq:nPCE-uNest} and ~\ref{eq:nPCE-uPush} can be removed for all $k$ with $LB^h < D_k < UB^h$. Moreover, the variables $u_k^h$ for every $k$ with $D_k < LB^h$ can also be removed (instead of fixed to zero) if constraints \eqref{eq:nPCE-zPush} are adapted as follows: Let $k'$ be the index of the smallest distinct distance $D_k$ for 
    which $D_k \geq LB^h$. The adapted constraints are
    \begin{align}
        D_{k'} + \sum_{k = k'}^{K} \left(D_k - D_{k-1}\right) u_k^h \leq z^H && \forall \inH
    \end{align}
    Naturally, we can also replace $K$ in this sum by the index of the largest distinct distance $D_k$ for which $D_k \leq UB^h$.
\end{proposition}

\begin{proof}
    Directly follows from the definitions of the variables and constraints. $\Box$
\end{proof}

\section{Implementation details} \label{sec:algorithm}
In this section we give implementation details of the exact solution algorithms we designed based on the formulations presented in the previous section. Depending on the underlying formulation, the solution algorithms are branch-and-bound (formulations \nPC and \nPCE) or branch-and-cut (formulation \nPCY) algorithms. We used C++ with CPLEX 22.1 as a framework to implement the algorithms.

We first describe a preprocessing algorithm and a starting heuristic. Both of these procedures can be used in combination with any of our solution algorithms. At the end of the section, we discuss the implementation details which are related to
the implementation of the B\&C used in conjunction with formulation \nPCY, i.e., details of the separation routine for inequalities \eqref{eq:nPCY-zPush}/\eqref{eq:nPCY-OPT}. In our computational study in Section \ref{sec:computational}, we provide results with various settings regarding turning on/off the ingredients described in this section.

\subsection{Preprocessing}\label{sec:preprocecssing}
In the propositions in Section \ref{sec:milp-formulations} we showed that good bounds on the value of the decision variables $z$ can improve the formulations. We developed thus a preprocessing phase to calculate good bounds. We obtain lower bounds on the $z^h$ by solving the \pCP with $p=p^h$, as the optimal objective function value of the \pCP is a valid lower bound on the $z^h$ variable.  
\begin{proposition} \label{prop:p-center}
    Let $d^{h*}$ be the optimal objective function value of \pCP\ for a certain $p^h$. Then $d^{h*}$ is a valid lower bound for the value of the decision variable $z^h$ of \npCPA.
\end{proposition}

\begin{proof}
    As the \npCPA for $|\mathcal{H}| = 1$ reduces to the \pCP, the \pCP is a relaxation of the \npCP.
    Therefore, the optimal objective value $d^{h*}$ of the \pCP for $p = p^h$ is a lower bound on the decision variable $z^h$ of \npCPA. $\Box$
\end{proof}

In our preprocessing algorithm we solve the \pCP for $p=p^h$ for all $\inH$ to obtain lower bounds for all $z^h$. We solve the \pCP by a modified version of the B\&C algorithm of \cite{GAAR2022} where we exploit that this algorithm also uses valid inequalities based on a lower bound on the objective function value: By solving the problems in a decreasing fashion according to $p^h$, we can always plug in the optimal value of the previous \pCP as a lower bound in the inequalities as the optimal solution value of the \pCP for a $p^{h'}>p^h$ is a lower bound for the optimal solution value of the \pCP for $p^h$. Thus, as the results in Section \ref{sec:settings} show, although we repeatedly solve an NP-hard problem, we obtain a computationally very efficient preprocessing algorithm. The obtained lower bounds on the values $z^h$ are then used according to Propositions \ref{prop:xy-lift}, \ref{prop:y-lift}, \ref{prop:npce-fixed}.

For our solution algorithms based on formulation \nPCE we also calculate upper bounds on $z^h$ in this preprocessing phase, as with this formulation upper bounds can also be exploited (see Proposition \ref{prop:npce-fixed}).

\begin{proposition}\label{prop:UB1}
    Given a an upper bound $UB$ on the objective function value of \npCPA and lower bounds $LB^h$ on the decision variables ${z}^h$, then
  \begin{equation}
        UB^h = UB - \sum_{h' \in \mathcal H: h' \neq h}LB^{h'}, \label{eq:UB}
    \end{equation}
    is a valid upper bound on the decision variable $z^h$ of the \npCPA for $\inH$.     
\end{proposition}

\begin{proof}
Let $z^*=\sum_{\inH} z^{h*}$ be the optimal objective function value of the considered instance, where $z^{h*}=d(\mathcal J^{h*})$ for a given optimal solution $(\mathcal J^{1*},\ldots,\mathcal J^{H*})$. Rewriting this equation gives $z^{h*}=z^*-\sum_{h' \in \mathcal H: h' \neq h} z^{h'*}$. The proposition follows from the fact that $UB\geq z^{*}$ and $LB^h\leq z^{h*}\leq UB^h$, $\inH$ by definition of the bounds. $\Box$
\end{proof}

There is also another way to calculate the upper bounds for the $z^h$-variables, as the next proposition shows.

\begin{proposition}\label{prop:UB2}
    Given an upper bound $UB$ on the objective function value of \npCPA and lower bounds $LB^h$ on the decision variables ${z}^h$, then
    \begin{equation}
        UB^h = \frac{UB - \sum_{h' = h+1}^{H}LB^{h'}}{h}, \label{eq:UB2}
    \end{equation}
    is a valid upper bound on the decision variable $z^h$ of the \npCPA for $\inH$. 
\end{proposition}

\begin{proof}
By the definition of the \npCPA we have that, given any upper bound $UB$, there must exist some $UB^h\geq LB^h$ with $UB=\sum_{h=1}^H UB^h$. Thus, for any time period $h \in \mathcal H$, we have $\sum_{h'=1}^{h}UB^{h'} \leq UB - \sum_{h'={h}+1}^{H}LB^{h'}$. Dividing both sides by $h$ gives us $(\sum_{h'=1}^{h}UB^{h'})/h \leq (UB - \sum_{h'={h}+1}^{H}LB^{h'})/h$. Moreover, as we must have $z^1 \geq \dots \geq z^h$ due to the definition of the \npCPA, it follows that $h z^h \leq \sum_{h'=1}^h z^{h'}$ and thus $z^h \leq (\sum_{h'=1}^{h}UB^{h'})/h$. $\Box$
\end{proof}

\begin{observation}
    Note that given any feasible solution, say $(\bar{\mathcal J^{1}},\ldots,\bar{\mathcal J^{H}})$, we have $UB=\sum_{\inH} d(\bar{\mathcal J^{h}})$. However, it is not true that $UB^h=d(\bar{\mathcal J^{h}})$ is valid. 
    For example, consider $\mathcal I=\{A,B\}$, $\mathcal J=\{a,b,c\}$ with $d_{Aa}=d_{Bb}=0$, $d_{Ab}=d_{Ba}=20$, $d_{Ac}=d_{Bc}=15$. Then a feasible solution is $(\{c\},\{c,a\})$ with $d(\{c\})=15$ and  $d(\{c,a\})=15$ (resulting in an objective function value of 30). Setting $UB^1=d(\{c\})=15$ is not valid, since there exist two optimal solutions for this instance, namely $(\{a\},\{a,b\})$ with $d(\{a\})=20$ and $d(\{a,b\})=0$ (resulting in an objective function value of 20) and the solution where $a$ and $b$ are switched, where we then have $\{b\}$ with $d(\{b\})=20$ as first component. Thus, setting $UB^1=d(\{c\})=15$ makes all optimal solutions infeasible. Note that this does not preclude that, given any $UB$ there must exist some $UB^h$ with $UB=\sum_{h=1}^H UB^h$ (as used in the proof of Proposition \ref{prop:UB2}), for the given feasible solution with value 30 in the example we could, e.g., set $UB^1=20$ and $UB^2=10$.
\end{observation}

\begin{proposition}
 The upper bounds $UB^h$ obtained by Propositions \ref{prop:UB1} and \ref{prop:UB2} are incomparable, i.e., given an upper bound $UB$ and lower bounds $LB^h$, $\inH$, either Proposition \ref{prop:UB1} or Proposition \ref{prop:UB2} can result in a better $UB^{h'}$ for a given $h'\in H$.
\end{proposition}

\begin{proof}
We proceed by first giving an example where Proposition \ref{prop:UB2} gives better upper bounds, and then provide an example where Proposition \ref{prop:UB1} gives better upper bounds.
\begin{itemize}
\item Suppose we are given any $UB$ and $LB^h=0$, $\inH$, clearly the latter is a set of valid lower bounds for any instance of \npCPA. Proposition \ref{prop:UB1} then gives $UB^h=UB$ for all $\inH$, while Proposition \ref{prop:UB2} gives $UB^h=UB/h$. 
\item Suppose we are given $UB=z^*$ and $LB^h=z^{h*}=d(\mathcal J^{h*})$ for a given optimal solution $(\mathcal J^{1*},\ldots,\mathcal J^{H*})$ for the considered instance. This is a valid upper bound and a set of valid lower bounds for any instance of \npCPA. Proposition \ref{prop:UB1} then gives $UB^h=z^{h*}$ for all $\inH$. Moreover, Proposition \ref{prop:UB2} gives $UB^h=(\sum_{h'=1}^h z^{h'*})/h$. Thus for any instance where there exists an optimal solution where for some $h'$ we have $z^{h'-1*}>z^{h'*}$ Proposition \ref{prop:UB2} gives a worse bound. $\Box$
\end{itemize}
\end{proof}

An upper bound $UB$ for use in the above propositions can be easily obtained during the preprocessing phase by using the optimal \pCP solution for $p=p^1$ as the following proposition shows.

\begin{proposition}\label{prop:upperbounds}
For a given instance of the \npCPA with a time horizon $\mathcal{H} = \left\{1, \dots, H\right\}$, let $\mathcal J^1$ be a set with $|\mathcal J^1|=p^1$ and let $z^{1*} = \max_{i \in \mathcal{I}}\min_{j \in \mathcal J^1} d_{ij}$. Then $Hz^{1*}$ is a valid upper bound on the objective function value of the instance.
\end{proposition}
\begin{proof}
Using $\mathcal J^1$ a feasible solution for the instance can be constructed by copying $\mathcal J^1$ to each $\mathcal J^h$, $h=\left\{2, \dots, H\right\}$ and then adding random facilities to these $\mathcal J^h$ until we have $|\mathcal J^h|=p^h$ for each $h$. Clearly the objective function value of this solution is at most $Hz^{1*}$.
$\Box$
\end{proof}

\subsection{Starting heuristics}\label{sec:startHeuristic}

We implemented three starting heuristics to obtain an incumbent solution to initialize our solution algorithms. The first heuristic is based on solving the \pCP for $p = p^H$ (which we do anyway in case the selected algorithmic setting also includes the preprocessing phase described in the previous section). We then use this solution as set $\mathcal J^H$. To construct the remaining solution, we proceed in a greedy fashion by iteratively removing $\delta_h=p^h - p^{h-1}$ facilities to obtain the set $\mathcal J^{h-1}$ from $\mathcal J^{h}$. It is easy to see that a solution constructed in this way is feasible. To select the $\delta_h$ facilities to remove, we enumerate all potential closings of size $\delta_h$ and take the one which gives the smallest increase in objective function value (i.e., the one where we obtain the $\mathcal J^{h-1}$ with smallest $\mathcal {R_A} (h-1,\mathcal J^{h-1})$), ties are broken by the least number of customers at distance of the objective value to the respective facility.

The second starting heuristic uses the optimal solutions of the \pCP for all $p \in \mathcal{P} = \left\{p^1, \dots p^H\right\}$. We count in how many solutions a facility is open. Then we take the $p^1$ facilities with the highest count to construct $\mathcal J^1$. To construct the further sets $\mathcal J^{h}$, $h=\{2,\ldots,H\}$ we proceed in a greedy fashion: Let $\mathcal J'$ be the set of currently open facilities (thus initially $\mathcal J'=\mathcal J^1$) and $\bar {\mathcal J}$ the set of facilities which occur in at least one of the optimal solutions of the \pCP. We add a facility $j^* \in \arg\min_{j' \in \bar {\mathcal J} \setminus \mathcal J'} \max_{i \in \mathcal I} \min_{j \in \mathcal J' \cup j'} d_{ij}$. Whenever $|\mathcal J'|=p^h$ for some $h=\{2,\ldots,H\}$ we set $\mathcal J^{h}=\mathcal J'$. 

The third heuristic is based on solving the \pCP for $p = p^1$. We then use this solution as set $\mathcal J^1$. To construct the remaining solution we proceed in a greedy fashion by iteratively adding $\delta_h=p^{h} - p^{h-1}$ facilities to obtain the set $\mathcal J^{h+1}$ from $\mathcal J^{h}$. It is easy to see that a solution constructed in this way is feasible. To select the $\delta_h$ facilities to add, we enumerate all potential openings of size $\delta_h$ and take the one which gives the largest decrease in objective function value (i.e., the one where we obtain the $\mathcal J^{h+1}$ with smallest $\mathcal {R_A} (h+1,\mathcal J^{h+1})$), ties are broken by the least number of customers at distance of the objective value to the respective facility.

\subsection{Implementation details of formulation \texorpdfstring{\nPCY}{\textit{(nPCA2)}}}

Similar to the first and third formulation, the second formulation \nPCY is a compact formulation and we could thus solve it with a branch-and-bound algorithm. However, following \cite{GAAR2022} (the work the formulation \nPCY is based on) we implemented a branch-and-cut algorithm where we start out with just the cardinality constraints \eqref{eq:nPCY-numOfFac} and nesting constraints \eqref{eq:nPCY-nested} and separate the (strengthened) inequalities \eqref{eq:nPCY-zPush}/\eqref{eq:nPCY-OPT}
 on-the-fly when they are violated. We have implemented this using the \texttt{UserCutCallback} and the \texttt{LazyConstraintCallback} of CPLEX, and use different separation schemes depending on the type of callback. This is done as the \texttt{UserCutCallback} is called when the solution of LP-relaxation at a node in the branch-and-bound tree is fractional, while the \texttt{LazyConstraintCallback} is called whenever a potential incumbent solution is encountered during the solution process. Next to solutions of the current LP-relaxation which fulfill all the integrality requirements this also includes solutions constructed by the internal heuristics of CPLEX. Such heuristic solutions are often quite different to the LP-relaxation solutions, thus the violated inequalities added for these solutions are not useful for improving the LP-relaxation value, however, they must be added to ensure correctness of the algorithm (e.g., otherwise a solution constructed by an internal CPLEX heuristics which sets the values of the $z$-variables wrongly could be accepted as feasible solution). In both cases, we consider all $h=\{1,\ldots,H\}$ independently, i.e., we call the separation scheme for each $\inH$.

 \paragraph{Separation routine implemented in the \texttt{LazyConstraintCallback}}
Let $(z^*,y^*)$ be the solution for which the separation routine is called. Since we are in the \texttt{LazyConstraintCallback} we know the integrality requirements are fulfilled by the solution as well as the cardinality constraints \eqref{eq:nPCY-numOfFac} and the nesting constraints \eqref{eq:nPCY-nested} (as they are kept in the initial relaxation), i.e., the values of the $y$-variables encode a feasible solution, but the values of the $z$-variables may measure the objective function value wrongly due to missing constraints \eqref{eq:nPCY-zPush}/\eqref{eq:nPCY-OPT}. Thus, for each $\inH$ we add the most violated constraint \eqref{eq:nPCY-zPush}/\eqref{eq:nPCY-OPT} (if any). This ensures that the values of the $z$-variables are correct. These most violated constraints can be easily found by inspection as $y^*$ encodes the set of open facilities.

\paragraph{Separation routine implemented in the \texttt{UserCutCallback}}

Our separation scheme is based on the \texttt{fixedCustomer} scheme of \cite{GAAR2022} (which in \cite{GAAR2022} performed better compared to a straightforward scheme of just adding the most violated constraint \eqref{eq:nPCY-zPush}/\eqref{eq:nPCY-OPT}, if there is any, for each customer in each separation round). It initially restricts the separation to a subset of customers $\hat{\mathcal I}$, which is
iteratively grown during the solution process. At each separation round, for all customers in $\hat{\mathcal I}$, the most violated
constraint \eqref{eq:nPCY-zPush}/\eqref{eq:nPCY-OPT}
(if any) is added. We describe at the end of the paragraph how this constraint can be found efficiently. The set $\hat{\mathcal{I}}$ is initialized using a greedy algorithm. A random customer is picked and added to the empty set $\hat{\mathcal I}$.
It is then iteratively grown by adding the customer $i \in \mathcal{I} \setminus \hat{\mathcal I}$ with the maximum distance to its closest
customer in the set $\hat{\mathcal I}$, until $|\hat{\mathcal I}| = p^H + 1$. Note that we exploit that in our test instances we have $\mathcal I=\mathcal J$ and thus we have distances between the customers available. If this would not be the case, a distance between $i,i'\in \mathcal I$ could be calculated as $d_{ii'}=\min_{\inJ} d_{ij}+d_{i'j}$.

After the initialization of $\hat{\mathcal I}$, the set is grown at each round of separation by determining the customer $i \in \mathcal{I} \setminus \hat{\mathcal I}$ with the largest violation. Additionally, if there are more than
\texttt{maxNoImprovementsFixed} consecutive iterations without an improvement in the lower bound we add more customers to $\hat{\mathcal I}$.
These customers are selected from a set $\overline{\mathcal{I}}$, which is initialized as $\overline{\mathcal{I}} = \mathcal{I} \setminus \hat{\mathcal I}$.
We then add the customer $i$ causing the largest violation of inequality \eqref{eq:nPCY-zPush}/\eqref{eq:nPCY-OPT} to $\hat{\mathcal I}$. After this, we remove all customers $i'$ with $d_{ii'}\leq LB^H$ from $\overline{\mathcal{I}}$, and then continue with this process of adding to $\hat{\mathcal I}$ and removing to $\overline{\mathcal{I}}$ until $\overline{\mathcal{I}}=\emptyset$. The idea behind this step is to get a diverse set of customers in $\hat{\mathcal I}$. Note that this step again exploits that in our instances we have $\mathcal I=\mathcal J$.

If the lower bound does not improve more than $\epsilon = 1e-5$ for \texttt{maxNoImprovements} separation
rounds, the separation is stopped at this B\&B node. Moreover, we perform at most \texttt{maxNumSepRoot} rounds of separation in the root-node of the branch-and-bound tree and at most \texttt{maxNumSepTree} rounds in all other nodes.

For a $h \in \mathcal H$, given a customer $i \in \mathcal I $ and a (partial) LP-relaxation solution $(z^{h*},y^{h*})$, the most violated inequality \eqref{eq:nPCY-OPT} can be found as follows (for details see \cite{Fischetti2017,GAAR2022}): Let $d_{ij}' = d_{ij}$ if $d_{ij} > LB^h$
and $d_{ij}' = LB^h$ otherwise (the separation of \eqref{eq:nPCY-zPush} can be done by setting $d_{ij}' = d_{ij}$ in all cases). Sort the facilities $\inJ$ in ascending order according to $d_{ij}'$. Only facilities
with $y_j^{h*} > 0$ need to be considered, as the others are not contributing to the potential violation. Furthermore, let us assume that the facilities are ordered
in the following way: $d_{i1} \leq \dots \leq d_{i|\mathcal{J}|}$. Let the \emph{critical facility} $j_i^h$ be the index such that $\sum_{j = 1}^{j^h_i -1} y^{h*}_j
    < 1 \leq \sum_{j = 1}^{j_i^h}  y^{h*}_j$. The inequality with the maximum violation is then
\begin{align*}
    z^h \geq d'_{ij^h_{i}} - \sum_{j':d'_{ij'} < d'_{ij^h_{i}}} (d'_{ij^h_{i}} - d'_{ij'})y_{j'}^h. 
\end{align*}

In our implementation, similar to \cite{GAAR2022}, the following parameter values were used for the fractional separation scheme:
\begin{itemize}
    \item \texttt{maxNoImprovements}: 100
    \item \texttt{maxNoImprovementsFixed}: 5
    \item \texttt{maxNumSepRoot}: 1000
    \item \texttt{maxNumSepTree}: 1
\end{itemize}

\paragraph{Primal heuristic}

We also implemented a primal heuristic for use with \nPCY within the \texttt{HeuristicCallback} of CPLEX. This is done, as the internal heuristics of CPLEX are not working as well as when using \nPC or \nPCE (where the full formulation is given to CPLEX) due to the on-the-fly addition of the constraints.. We retrieve the (partial) optimal solution $y^*$ of the LP-relaxation at the current branch-and-bound node and for each location $\inJ$ sum $y_j^{h*}$ up over all time periods $\inH$. We then sort the obtained values in descending order. This gives us the facilities which are open "the most" over the time periods in the relaxation-solution. Then we construct a feasible solution for the \npCPA by taking the first $p^1$ facilities and set the corresponding decision variables to one. Then we take the first $p^2$ facilities and so on, until we have done this for all time periods $\inH$. The nesting constraint holds per construction, as the facilities we open in $p^h$ are a subset of the facilities opened in $p^{h+1}$.

\section{Computational results} \label{sec:computational}

The experiments were run on a single core of an Intel Xeon X5570 machine with 2.93 GHz. All CPLEX parameters were left at the default values. The time limit was set to 3600 seconds and the memory limit to 8 GB. 

\subsection{Instances}

In our computational study we use two sets of benchmark instance that have been used for the \pCP in \cite{Elloumi2018,calik2013double,chen2009new,contardo2019scalable,GAAR2022}:
\begin{itemize}
    \item \pmed: This instance set contains 40 instances. For all instances all customer demand points are
          also potential facility locations, i.e., $\mathcal I = \mathcal J$ holds. We denote this set as nodes $V$ in the following. The number of nodes ($|V|$) ranges between 100 and 900. The instances are given as graphs, and the distances $d_{ij}$ are the shortest-path distances between
          $i,j\in V$ in the graph. The instances also contain values for $p$, with these values ranging from 5 to 200 depending on the instance (for the concrete value of $p$ see e.g.,  Table \ref{tab:nested_sH_pmed}).    
          As values for $\mathcal P$ for these instances we take $\{p,p+1,p+2\}$ for the $p$ given in the instance.
    \item \tsplib: This is an instance set that was originally introduced for the traveling salesperson problem in \cite{Reinelt1991}. Again, $\mathcal I =\mathcal J =V$. We consider in our computational study a subset of instances where $|V|$ ranges between 51 and 1002. The instance names contain the exact values of $|V|$, e.g., for instance \texttt{eil51} $|V| = 51$. 
    The instances contain the two-dimensional coordinates for each point, and the distance is calculated as the Euclidean distance. Following the literature on the \pCP, the distances were rounded to the nearest integer. In these instances, we use $\mathcal P=\{4,5,6\}$.
\end{itemize}

\subsection{Results}\label{sec:settings}
In this section, we compare the following three different settings for each of our three formulations. 
\begin{itemize}
    \item \noPP: The basic setting, where we do not run our preprocessing or use the starting heuristics. In particular, this means the following for each formulation.
    \begin{itemize}
        \item \nPC: Proposition \ref{prop:xy-lift} not used.
        \item \nPCY: Proposition \ref{prop:y-lift} not used. The \texttt{HeuristicCallback} is used (as this primal heuristic was implemented to mitigate the effects of CPLEX not having the complete formulation available in case of using formulation \nPCY).
        \item \nPCE: Proposition \ref{prop:npce-fixed} not used.
    \end{itemize}
    \item \PP: Preprocessing is used to obtain lower and upper bounds and these bounds are used according to Propositions \ref{prop:xy-lift}, \ref{prop:y-lift}, \ref{prop:npce-fixed}.
    \begin{itemize}
        \item \nPC: Proposition \ref{prop:xy-lift} is used.
        \item \nPCY: Proposition \ref{prop:y-lift} is used. The \texttt{HeuristicCallback} is used.
        \item \nPCE: Proposition \ref{prop:npce-fixed} is used.
    \end{itemize}
    \item \sH: \PP plus the starting heuristics described in \ref{sec:startHeuristic}
\end{itemize}

In Figure \ref{fig:settings} cumulative distribution functions (CDF) of the runtimes and optimality gaps are plotted for \nPC, \nPCY and \nPCE respectively, with above settings and over both sets of instances. The optimality gap is calculated as $100\frac{z^*-\bar z}{z^*}$ where $z^*$ is the value of the best solution found and $\bar z$ is the value of the lower bound.

\begin{figure}
    \centering
    \begin{subfigure}[t]{.5\textwidth}
        \centering
        \includegraphics[width=\linewidth]{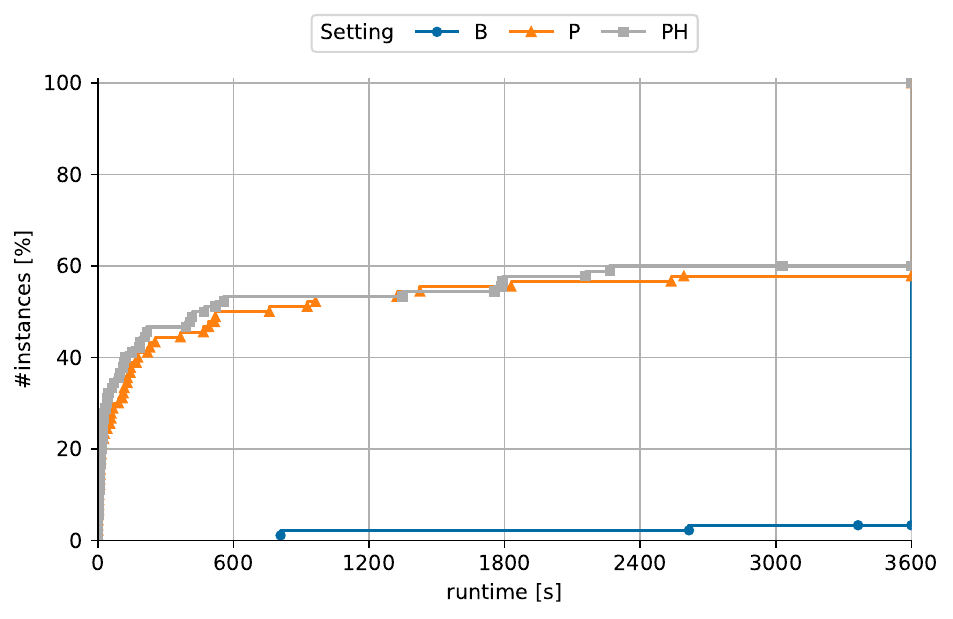}
        \caption{Runtimes for \nPC}
        \label{fig:xy-runtime}
    \end{subfigure}%
    \begin{subfigure}[t]{.5\textwidth}
        \centering
        \includegraphics[width=\linewidth]{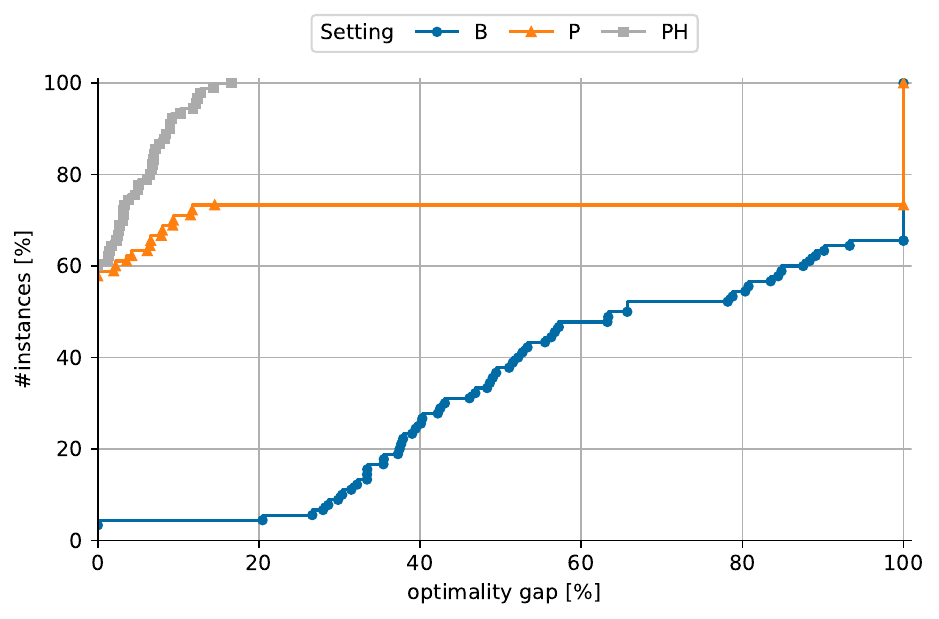}
        \caption{Optimality gaps for \nPC}
        \label{fig:xy-gap}
    \end{subfigure}%#
    
        \begin{subfigure}[t]{.5\textwidth}
        \centering
        \includegraphics[width=\linewidth]{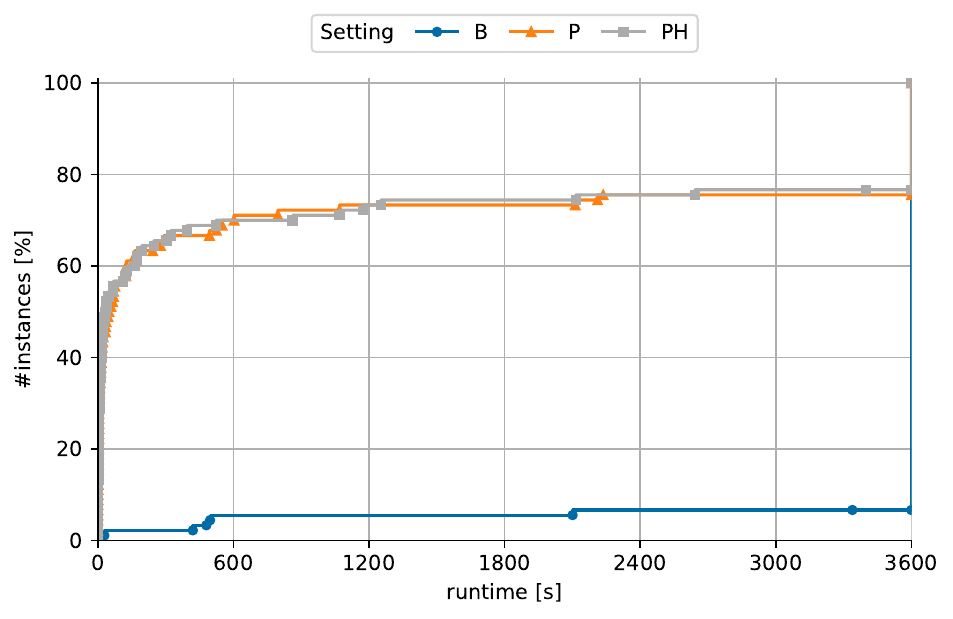}
        \caption{Runtimes for \nPCY}
        \label{fig:y-runtime}
    \end{subfigure}%
    \begin{subfigure}[t]{.5\textwidth}
        \centering
        \includegraphics[width=\linewidth]{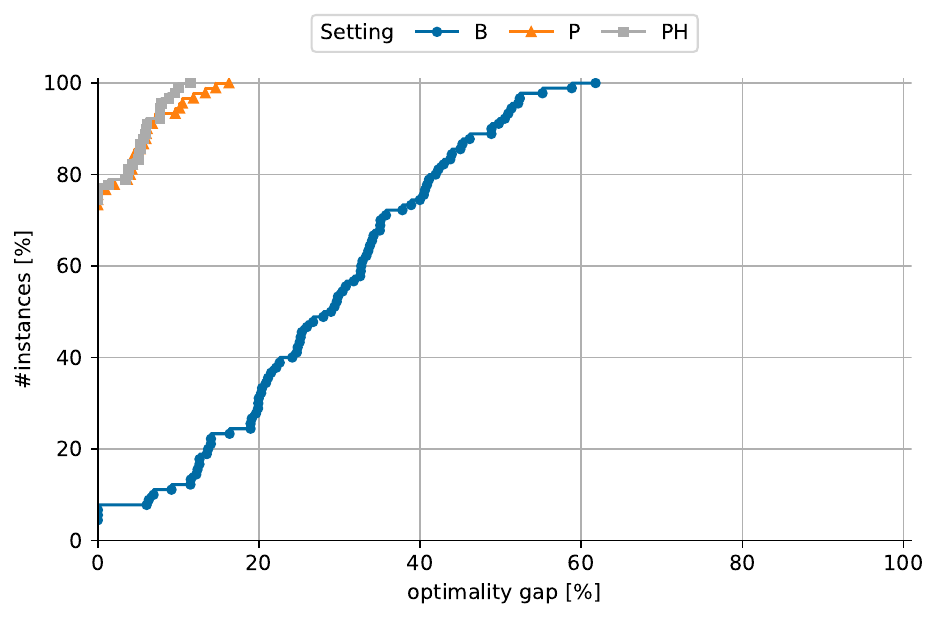}
        \caption{Optimality gaps for \nPCY}
        \label{fig:y-gap}%
    \end{subfigure}

    \begin{subfigure}[t]{.5\textwidth}
        \centering
        \includegraphics[width=\linewidth]{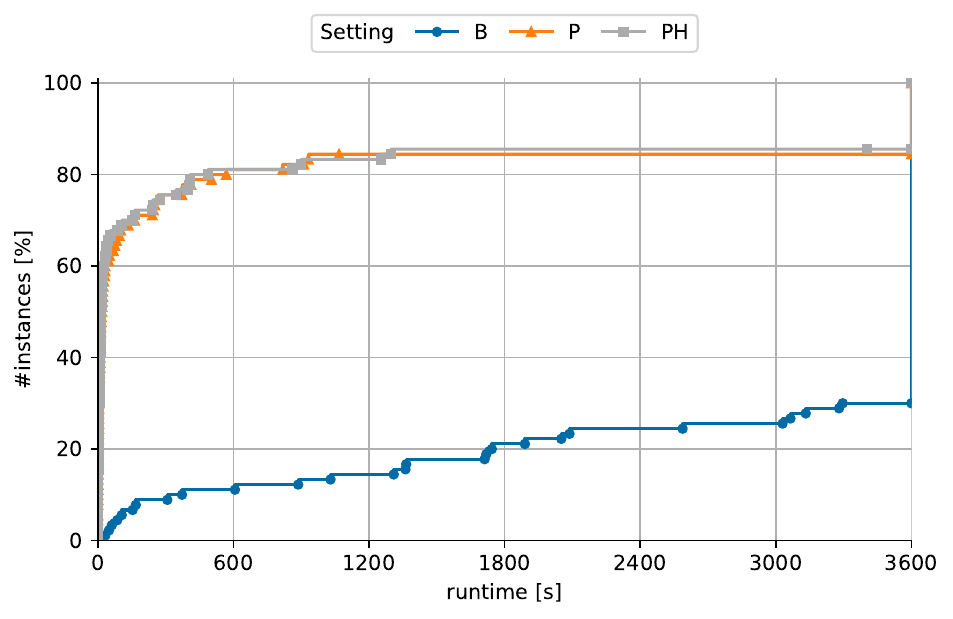}
        \caption{Runtimes for \nPCE}
        \label{fig:u-runtime}
    \end{subfigure}%
    \begin{subfigure}[t]{.5\textwidth}
        \centering
        \includegraphics[width=\linewidth]{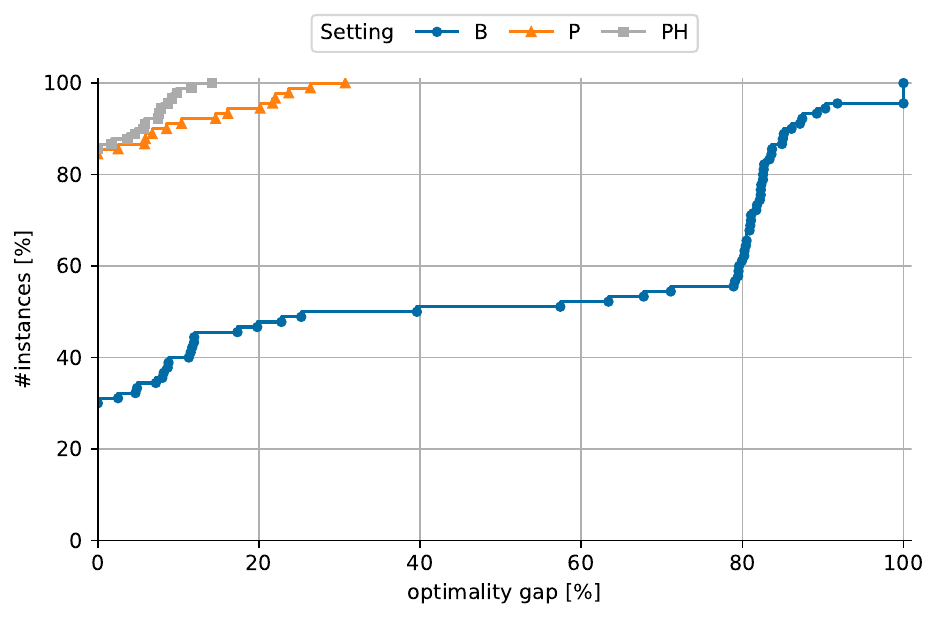}
        \caption{Optimality gaps for \nPCE}
        \label{fig:u-gap}
    \end{subfigure}
    \caption{Runtimes and optimality gaps for different settings and formulations}
    \label{fig:settings}
\end{figure}

In the figure, it can be seen that preprocessing is very important regardless of the used formulation. Depending on the formulation using preprocessing increases the number of instances which can be be solved within the given time from unter 5\% to nearly 60\% (\nPC), from around 10\% to nearly 80\% (\nPCY) and from around 30\% to over 80\%(\nPCE). This is not surprising in case of \nPCE and \nPCY, as for the \pCP it is known from \cite{GAAR2022} that the \pCP-version of Propositions \ref{prop:xy-lift}/\ref{prop:y-lift} is crucial to get better LP-relaxation bounds if such type of formulation is used. Interestingly, preprocessing also has a rather large effect when using formulation \nPCE, which is a formulation where the \pCP-counterpart is actually known to be stronger than the \pCP-counterpart of the \nPC/\nPCY formulations. 

The effect of using the starting heuristics is less pronounced, however for \nPC and \nPCE the setting \sH performs slightly better than \PP with regard to the optimality gaps. Overall, formulation \nPCE with setting \sH performs best, managing to solve over 85\% of instances to optimality within the given time limit, and having an optimality gap of at most around 15\% for the instances not solved to optimality. The performance of \nPCY with either setting \PP or \sH is not far off, the largest optimality gap is actually similar to \nPCE with \sH, however, it manages to only solve 75\% of instances to optimality within the time limit. 

Table \ref{tab:nested_sH_pmed} shows detailed results for all three formulations for the instance set \pmed and the best-performing setting \sH. The table gives the instance index ($idx$), the number of nodes ($|V|$), the value of $p$ (recall that $\mathcal P=\{p,p+1,p+2\}$), the runtimes (t[s]), the value of the best found solution (UB), the optimality gap ($g[\%]$) and the number of branch-and-bound nodes (bnb). The fastest runtime is highlighted in bold and TL is shown if the instance is not solved to optimality within the time limit. 

\input{media/pmed_sH_detailed.tex}

All 40 instances are solved to optimality with formulations \nPCY and \nPCE (while for 11 instances formulation \nPC runs into the time limit). The formulation \nPCY is able to solve 24 formulations faster than \nPCE, while \nPCE is faster for 16 instances. We note that for most of the instances these differences are minimal, i.e., within a few seconds or less, however, there are also instances, where this differences is quite pronounced, e.g., for instance 6 \nPCE is about 300 seconds faster, while for instance 32, \nPCY is about 900 seconds faster. The dependency on the size of the set of distinct distances of formulation \nPCE could be a probable cause for these mixed results. Moreover, most of the instances are solved within the root-node, indicating that \nPCY and \nPCE (including the propositions for the strengthening) seem both to be very strong formulations. Also for \nPC the number of branch-and-bound nodes is often zero, however sometimes the algorithms runs into the time limit while still being at node zero. A reason for this could be that the formulation \nPC has many more variables than \nPCY and \nPCE, thus solving the LP-relaxation becomes much more time-consuming. 

In Table \ref{tab:nested_sH_tsplib} similar results are presented for the \tsplib instances. The table gives the filenames (column $f$, which includes the number of nodes $|V|$), the runtimes (t[s]), the value of the best found solution (UB), the optimality gap ($g[\%]$) and the number of branch-and-bound nodes (bnb). The fastest runtime is highlighted in bold and TL is shown if the instance is not solved to optimality within the time limit.  Recall that for these instances $\mathcal P=\{4,5,6\}$. 

\begin{small}
    \input{media/tsplib_sH_detailed.tex}
\end{small}

This instance-set turns out to be more difficult compared to \pmed. A reason for this could be that the distances in this set are based on (rounded) Euclidean distances and not shortest-paths in graphs, resulting in a greater range of distances and thus a greater range potential optimal solutions. For formulation \nPCE this also means that there are more $u$-variables. Using formulation \nPC 25 out of the 50 instances could be solved to optimality, with \nPCY 29 and with \nPCE 37. For most of the instances \nPCE outperforms \nPCY, however, there is also instance \texttt{pr439} which is solved to optimality by \nPCY but not by \nPCE. We also see that \nPCE sometimes runs into the time limit while still being at the root node, which indicates that for these instances the LP-relaxations are more time consuming.

\subsection{Managerial insights}

Next, we focus on the impact of nesting on the objective function value and the structure of the obtained solutions.
In Figure \ref{fig:rel-reg-opt} a CDF-plot of the \emph{relative cost of nesting (RC)} defined as $\frac{\sum_{\inH} \mathcal R_A(h,\mathcal J^{h*})}{\sum_{\inH} d^{h*}}$ for the optimal/best found solution $(\mathcal J^{1*}, \ldots, \mathcal J^{h*})$ of each instance, is given. This value is below 2.5 \% for instances for the set \pmed and below 15 \% for instances from the set \tsplib. This means that for these instances the nested solution is not much worse than the corresponding optimal \pCP-solutions.
Figure \ref{fig:nPC-rel-reg} provides a plot of the \emph{relative increase in number of open facilities (RI)} defined as $\frac{|\mathcal J^*|-p^H}{p^H}$ for each each instance, where $\mathcal J^*=\{j \in \mathcal J: j \text{ was open in the optimal solution CPLEX found when solving the $\pCP$ for a } p \in \mathcal P \}$\sloppy\ for each instance. For instance set $\pmed$ there is a minimal RI of around 25\% and for instance set $\tsplib$ the minimal RI is over 60\%. For both instance sets the maximal RI is around 150\%. 

\begin{figure}[h!tb]
    \begin{subfigure}[t]{.5\textwidth}
        \centering
        \includegraphics[width=\linewidth]{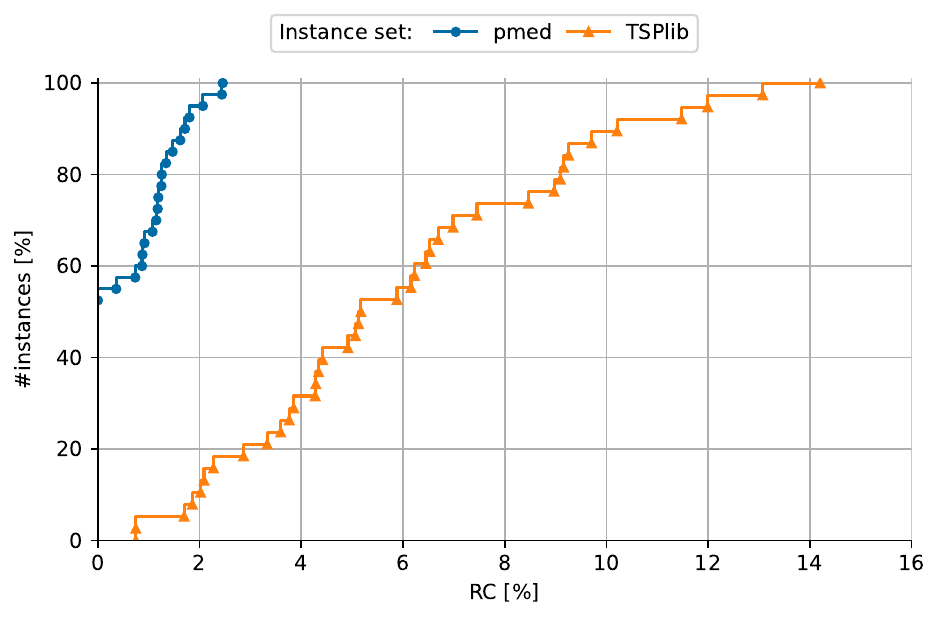}
        \caption{Relative cost of nesting}
        \label{fig:rel-reg-opt}
    \end{subfigure}%
    \begin{subfigure}[t]{.5\textwidth}
        \centering
        \includegraphics[width=\linewidth]{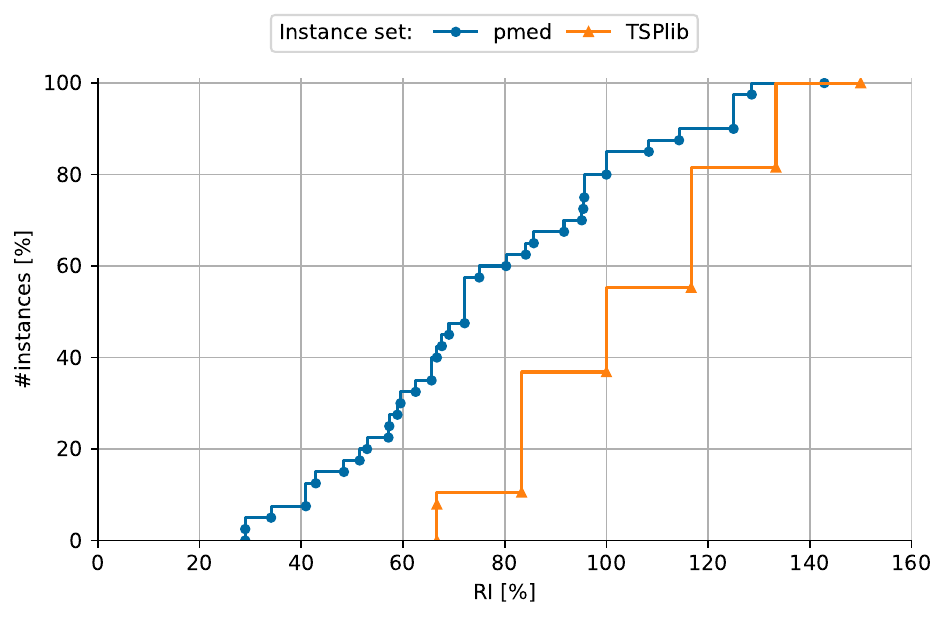}
        \caption{Relative increase in number of open facilities}
        \label{fig:rel-reg-fac}
    \end{subfigure}
    \caption{Comparison of the relative regrets of the optimal solution value and number of open facilities for \pCP\ and \npCP}
    \label{fig:nPC-rel-reg}
\end{figure}

Thus, overall we see that for the considered instances, using the concept of nesting when dealing with the \pCP in a multi-period settings results in modest increase in solution cost, while the number of different open facilities in the optimal solutions obtained when just considering the \pCP individually over the time horizon without nesting is often considerably larger. This means that the consistent solutions provided by solving the \npCPA can be very attractive for decision makers.

\section{The nested \texorpdfstring{$p$-center}{p-center} problem with the minimizing the maximum relative regret objective} \label{sec:regret}

In this section, we first present two MILP formulations for the \npCPR, followed by implementation details and a computational study. 

\subsection{Mixed-integer linear programming formulations}\label{sec:milpr}
The following formulation is based on \nPCY and uses similar decision variables. To measure the maximal regret over the time periods we introduce the additional continuous decision variable $w$.
\begin{subequations}\label{eq:nPCYR}
    \begin{align}
         & \nPCYR &  & \min         & w                      &                                                                      &  & \label{eq:nPCYR-objective}                                                 \\
         &        &  & \text{ s.t.} & \sum_{\inJ}    y_{j}^h & =     p^h                                                            &  & \forall \inH                                     \label{eq:nPCYR-numOfFac} \\
         &        &  &              & z^{h}                  & \geq  d_{ij} - \sum_{j':d_{ij'} < d_{ij}} (d_{ij} - d_{ij'})y_{j'}^h &  & \forall \inI, \inJ, \inH                         \label{eq:nPCYR-zPush}    \\
         &        &  &              & w                      & \geq  \frac{z^h - d^{h*}}{d^{h*}}                                    &  & \forall \inH                                     \label{eq:npCYR-regret}   \\
         &        &  &              & y_j^{h}                & \geq  y_j^{h-1}                                                      &  & \forall \inJ, \inH \setminus \left \{1 \right \} \label{eq:nPCYR-nested}   \\
         &        &  &              & y_{j}^h                & \in   \{\text{0, 1}\}                                                &  & \forall \inJ, \inH                               \label{eq:nPCYR-binY}     \\
         &        &  &              & z^{h}                  & \in   \mathbb R_{\ge 0}                                              &  & \forall \inH                                     \label{eq:nPCYR-nonnegZ}  \\
         &        &  &              & w                      & \in   \mathbb R_{\ge 0}                                              &  & \label{eq:nPCYR-nonnegw}
    \end{align}
\end{subequations}
The formulation is very similar to \nPCY and only contains an additional set of constraints \eqref{eq:npCYR-regret}, that ensures that decision variable $w$ 
is at least as large as the largest relative regret of all time periods. The objective function now minimizes the decision variable $w$ and not the sum of the decision variables $z^h$.

Similar as for formulation \nPCY for the \npCPR, inequalities \eqref{eq:nPCYR-zPush} can be strengthened. However, note that this time the conditions on the lower bound $LB^h$ are a bit different compared to before (i.e., in the propositions for the formulations for \npCPA the condition was connected with the lower bound of the optimal solution value  ).

\begin{proposition}
   Let $LB^h \geq 0$ be a valid lower bound on the optimal objective function value of the \pCP for $p^h=p$ for a given $\in H$. Then
      \begin{equation}
        z^h \geq \max \left\{LB^h, d_{ij}\right\} - \sum_{j':d_{ij'} < d_{ij}} \left(\max \left\{LB^h, d_{ij}\right\} - \max \left\{LB^h, d_{ij'}\right\}\right)y_{j}^h \label{eq:L-OPTR2}
    \end{equation}
    is valid for \nPCYRN, i.e. every feasible solution of \nPCYRN\ fulfills \eqref{eq:L-OPTR2}.
\end{proposition}

\begin{proof}
Note that the value of $z^h$ must always be at least the value of the the optimal objective function value of the \pCP for $p^h=p$. The validity then follows from similar arguments as in the proof of Proposition \ref{prop:xy-lift}. $\Box$
\end{proof}

It is possible to remove the variables $z^h$ from \nPCYR to get a new formulation which we denote as \nPCYRN. 
\begin{subequations}\label{eq:nPCYRN}
    \begin{align}
         & \nPCYRN  && \min            & w                    &                                                                                                     && \label{eq:nPCYRN-objective}                      \\
         &          && \text{ s.t.}    & \sum_{\inJ}y_{j}^h   &= p^h                                                                                                && \forall \inH \label{eq:nPCYRN-numOfFac}          \\
         &          &&                 & w                    &\geq \frac{1}{d^{h*}}\left(d_{ij}-\sum_{j':d_{ij'}<d_{ij}}\left(d_{ij}-d_{ij'}\right)y_{j'}^h\right) -1&& \forall \inI, \inJ, \inH \label{eq:nPCYRN-wPush} \\
         &          &&                 & y_j^{h}              &\geq y_j^{h-1}                                                              && \forall \inJ, \inH \setminus \left \{1 \right \} \label{eq:nPCYRN-nested} \\
         &          &&                 & y_{j}^h              &\in \{\text{0, 1}\}                                                                                  && \forall \inJ, \inH \label{eq:nPCYRN-binY}        \\
         &          &&                 & w                    &\in \mathbb R_{\ge 0}                                                                                && \label{eq:nPCYRN-nonnegw}
    \end{align}
\end{subequations}
In this new formulation, the constraints \eqref{eq:nPCYR-zPush} and \eqref{eq:npCYR-regret} are combined into \eqref{eq:nPCYRN-wPush}. This allows us to omit the decision variables $z$. For correctness, consider the following proposition.
\begin{proposition}\label{prop:reform}
    The formulation \nPCYRN\ is a valid formulation for the \npCPR.
\end{proposition}
\begin{proof}
    Constraints \eqref{eq:npCYR-regret} can be rewritten into
    \begin{align*}
         & w + 1 \geq \frac{1}{d^{h*}} z^h & \forall \inH.
    \end{align*}
    As this has to hold for every $\inH$ and therefore every $z^h$ we can further reformulate it to
    \begin{align*}
         & w \geq \frac{1}{d^{h*}} \left(d_{ij} - \sum_{j':d_{ij'} < d_{ij}} \left(d_{ij} - d_{ij'}\right)y_{j'}^h \right) -1 & \forall \inI, \inJ, \inH
    \end{align*}
   $\Box$
\end{proof}

Inequalities \eqref{eq:nPCYRN-wPush} can be strengthened using a lower bound $LB$ on the objective function value of the \npCPR as follows.

\begin{proposition}
\label{lemma:optr}
    Let $LB \geq 0$ be a valid lower bound on the optimal objective function value of the \npCPR. Then for any $\inI, \inJ$ and $\inH$ the inequality
    \begin{equation}
        w \geq \frac{1}{d^{h*}} \left(\max \left\{LB d^{h*}, d_{ij}\right\} - \sum_{j':d_{ij'} < d_{ij}} \left(\max \left\{LB d^{h*}, d_{ij}\right\} - \max \left\{LB d^{h*}, d_{ij'}\right\}\right)y_{j'}^h \right)-1 \label{eq:L-OPTR} %\tag{\emph{L-OPTR}}
    \end{equation}
    is valid for \nPCYRN, i.e. every feasible solution of \nPCYRN\ fulfills \eqref{eq:L-OPTR}.
\end{proposition}
\begin{proof}
The proof follows the same arguments as the proof of Proposition \ref{prop:xy-lift}. 
\end{proof}

Note that compared to all previous propositions involving a lower bound, Proposition \eqref{lemma:optr} allows something slightly more powerful in practice: We can use the current bound of the LP-relaxation of \nPCYRN (with \eqref{eq:L-OPTR}) as $LB$ to iteratively create stronger versions of \eqref{eq:L-OPTR} (until at some points we get convergence, i.e., the LP-relaxation bound does not change anymore). This is similar to the situation for the lifted optimality cuts for the \pCP in \cite{GAAR2022}. This was not possible in case of the presented formulations and propositions for the \npCPA, as for these, the needed lower bounds are for the $z^h$-variables and the LP-relaxation value is the sum of the values of the $z^h$-variables in the relaxation, thus we do not get an individual bound on any $z^h$ from the LP-relaxation. In this case it may be tempting to think about using the LP-relaxation value of a $z^h$-variable to use as $LB^h$, but this is not valid, as there could be multiple optimal solutions to the LP-relaxation and in each of these optimal solutions an individual $z^h$ may take different values (and only the minimal of these values would be a valid lower bound for use as $LB^h$).

\subsection{Implementation details} \label{sec:npcr-implementation}
Note that in contrast to the \npCPA where running the preprocessing scheme (resp., solving the \pCP for $p \in \mathcal P$) is optional, for the \npCPR we need the values $d^{h*}$ due to constraints \eqref{eq:nPCYR-zPush}/\eqref{eq:nPCYRN-wPush} (and the strengthened version \eqref{eq:L-OPTR}). Thus, we always run the preprocessing scheme described in Section \ref{sec:preprocecssing} when solving the \npCPR.

The close relationship between \nPCY, \nPCYR and \nPCYRN allows to use the same separation algorithm for all three formulations with minimal 
variation. For \nPCYR, as the constraints \eqref{eq:nPCY-zPush} and \eqref{eq:nPCYR-zPush} (and their strengthened versions) are the same, we can use the exact same separation algorithm. We use the optimal solution values of the \pCP obtained in the preprocessing scheme as $LB^h$.
For constraints \eqref{eq:nPCYRN-wPush}/\eqref{eq:L-OPTR} of formulation \nPCYRN we have to divide the right-hand side by $d^{h*}$ and change the left-hand side to $w$. Since this division is just a constant-scaling, we can again use the same algorithm to find the inequality giving the maximum violation of \eqref{eq:nPCY-zPush} for a given $\inH$ and then scale by $d^{h*}$ and check if the obtained right-hand-side is larger than the current LP-relaxation value of $w$. We use the trivial lower bound of $1$ as initial $LB$ for use in \eqref{eq:L-OPTR}. Moreover, we also implemented a more sophisticated setting which iteratively updates the $LB$ used in \eqref{eq:L-OPTR} which is described next.

\paragraph{Enhanced separation procedure for \eqref{eq:L-OPTR}}

In this enhanced separation procedure, we make use of the fact discussed at the end of Section \ref{sec:milpr}, i.e., we use the objective function value of the LP-relaxation of the current branch-and-bound node as lower bound $LB$. In fact, this lower bound can be further improved in some cases: Note that \npCPR can only take values $d_{ij}/d^{h*}-1$ for some $\inI$, $\inJ$, $\inH$. Moreover, all the distances $d_{ij}$ in the instances we consider are integer. As a consequence, the smallest possible change in the objective function value caused by period $h$ is $1/d^{h*}$, which we denote as \emph{step size} in the following. The objective function value obtained by the LP-relaxation (i.e., the current lower bound) though does not follow this step size. This allows us to increase the lower bound $LB$ to the next smallest \emph{step}. Consider the following example:
\begin{example}
    Let the $d^{1*} = 25 \,(0.04), d^{2*} = 20 \,(0.05), d^{3*} = 10 \,(0.1)$,
    where the resulting step sizes are given in the braces. Furthermore, let the objective value of the LP-relaxation be 0.07. Thus the lower bound can be increased to 0.08, as the optimal objective function value
    can only take multiples of any step size and since it can not be lower than 0.07.
\end{example}
This potentially increased lower bound is calculated each time when the \texttt{UserCutCallback} or \texttt{Lazy\-Con\-straint\-Call\-back} is entered. We use the enhanced separation procedure not only in the root-node of the branch-and-bound tree, but also in other nodes of the tree. Naturally the obtained lower bound at a branch-and-bound node is only feasible for the subtree originating at this node. Thus when using this procedure any violated valid inequality \eqref{eq:L-OPTR} is added with the \texttt{local} option to CPLEX, which means it is only added for the current branch-and-bound and the nodes in the subtree originating at it.

\paragraph{Starting heuristics and primal heuristic}
We used the same starting heuristics and primal heuristics as described in Section \ref{sec:startHeuristic}, as we observed that the solutions obtained when considering the absolute regret objective are usually also quite good for the relative regret objective, thus we did not implement more specifically tailored heuristics.

\subsection{Results}\label{sec:reg_results}
In this section, we compare the following three settings for solving the \npCPR. The runs were made on the same machine as the computational study on the \npCPA in Section \ref{sec:computational}.

\begin{itemize}
    \item \R: The formulation \nPCYR with strengthened inequalities \eqref{eq:L-OPTR} and implementation as described in Section \ref{sec:npcr-implementation}, i.e., including the starting and primal heuristics.
    \item \RN: The formulation \nPCYRN with strengthened inequalities \eqref{eq:L-OPTR2} and implementation as described in \ref{sec:npcr-implementation} except the enhanced separation procedure.
    \item \RNL: \RN including the enhanced separation procedure. 
\end{itemize}

The results are shown in cumulative distribution function plots in \ref{fig:regret-results}, where in \ref{fig:regret-runtimes} the runtimes are plotted and in \ref{fig:regret-gap} the optimality gaps.

\begin{figure}
    \centering
    \begin{subfigure}{.5\textwidth}
        \centering
        \includegraphics[width=\linewidth]{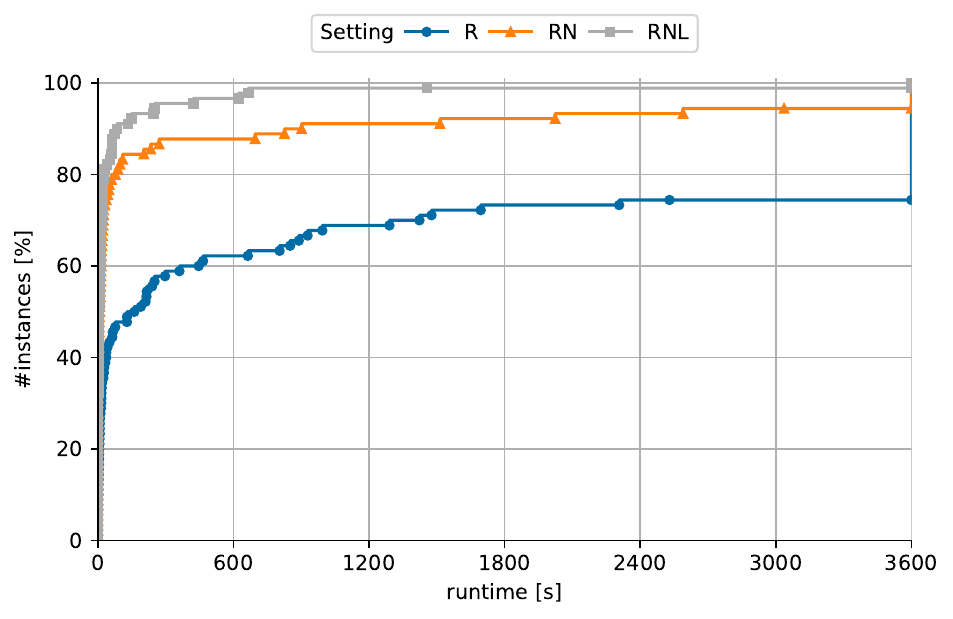}
        \caption{Runtime comparison}
        \label{fig:regret-runtimes}
    \end{subfigure}%
    \begin{subfigure}{.5\textwidth}
        \centering
        \includegraphics[width=\linewidth]{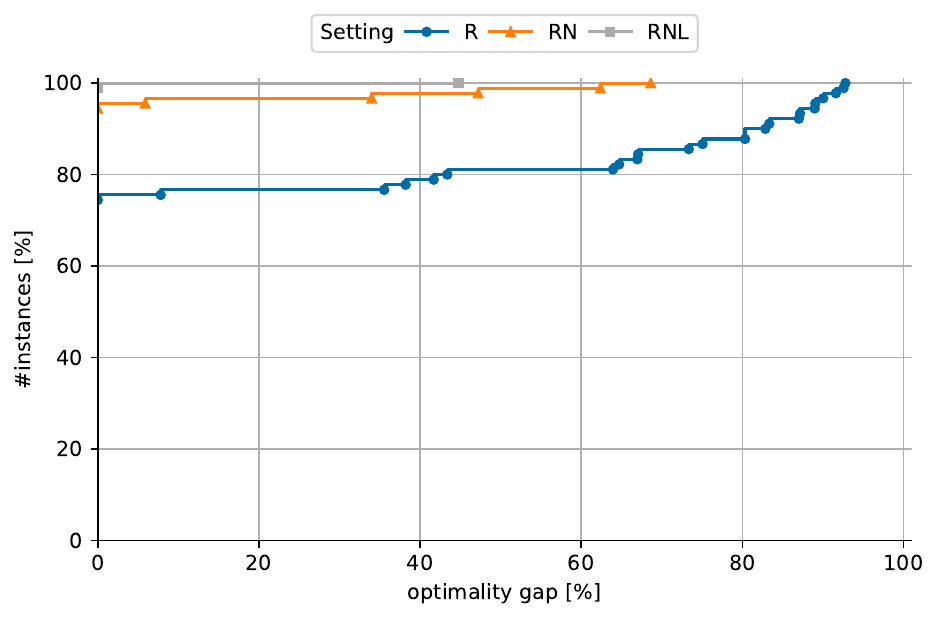}
        \caption{Optimality gaps comparison}
        \label{fig:regret-gap}
    \end{subfigure}
    \caption{Comparison of the different settings}
    \label{fig:regret-results}
\end{figure}

The figures show that both settings using formulation \nPCYRN outperform the setting using formulation \nPCYR. Moreover, using the enhanced separation procedure improves the performance, only one instance was not solved to optimality within the time limit of 3600 seconds. In Tables \ref{tab:pmed_regret} and 
\ref{tab:tsplib_regret}
the detailed results are shown. We see that with \RN and \RNL most of the instances can be solved within the root node and the improvement the enhanced separation procedure provides is visible in instances like 17, 22 and others of set \pmed. Interestingly, for instance set \tsplib the effect of the enhanced sepration procedure is less pronounced. For the \pmed instances it can be seen that the objective function value is at most 0.04 which means that for these instances the cost of nesting is very small. On the other hand, for the \tsplib instances, the range of objective function values is much larger, going up to 0.39 for instance \texttt{pr107}.

\input{media/pmed_regret_detailed}
\input{media/tsplib_regret_detailed}

\section{Conclusion and outlook} \label{sec:conclusion}
In this work, we introduce the nested $p$-center problem, which is a multi-period variant of the well-known $p$-center problem. The use of the \emph{nesting} concept allows to obtain solutions, which are consistent over the considered time horizon, i.e., facilities which are opened in a given time period stay open for the later time periods. This is important in real-life applications, as closing (and potential later re-opening) of facilities between time periods can be quite undesirable due to i) resulting monetary and environmental costs and ii) non-acceptance of such solutions by decision makers and the general public, as they are usually not consistent with intuition. 

Following \cite{McGarvey2022} who recently studied a nested version of the $p$-median problem, we consider two different versions of our problem, with the difference being the objective function. The first version considers the sum of the absolute regrets over all time periods, and the second version considers minimizing the maximum relative regret over the time periods, where the absolute regret of a time period is defined as the additional cost incurred due to the nesting constraint compared to just solving the problem without the nesting constraint.

We present three mixed-integer programming formulations for the version with absolute regret objective and two formulations for the version with relative regret objective. For all formulations, we present valid inequalities. Based on the formulations and the valid inequalities, we develop branch-and-bound and branch-and-cut solution algorithms. These algorithms include a preprocessing procedure that exploits the nesting property and also begins heuristics and primal heuristics.

We conducted a computational study on instances from the literature for the $p$-center problem which we adapted to our problems. The computational study illustrates that the preprocessing phase brings a significant
improvement in performance. Moreover, using our approaches, most of the instances could be solved within a time limit of 3600 seconds, showing that incorporating a multi-period perspective into the $p$-center problem is computationally feasible. We also analysed the effect of nesting on the solution cost and the number of open facilities and found that using the nesting concept, the relative increase in the solution cost stays within 15 \% compared to the optimal $p$-center solutions of the individual time periods, while the number of open facilities increases by up to 150 \% when considering the individual $p$-center solutions against the obtained nested solution.

There are several avenues for future work: One could try to further exploit the nesting property by, e.g., implementing cut pooling, where violated inequalities obtained when solving the $p$-center problem in preprocessing are stored and added to the nested $p$-center problem to start with an initial set of inequalities. Moreover, to be able to tackle larger-scale instances of the problems, it could be interesting to develop (meta-)heuristics. The nesting constraint could offer an interesting challenge in this regard. It could also be fruitful to apply the nesting concept to different versions of the $p$-center problem, in particular to stochastic and robust versions of it.

\section*{Acknowledgments}

This research was funded in whole, or in part, by the Austrian Science Fund 
(FWF)
[P 35160-N]. For the purpose of open access, the author has applied a CC BY 
public 
copyright licence to any Author Accepted Manuscript version arising from this submission. It is also supported by the Johannes Kepler University Linz, Linz Institute
of Technology (Project LIT-2021-10-YOU-216).

%Bibliography
\bibliography{references}  
\bibliographystyle{elsarticle-harv}

\end{document}

%% file: media/pmed_sH_detailed.tex
\begingroup\fontsize{6}{10}\selectfont
\begin{longtable}{lrrrrrrrrrrrrrr}
\caption{\label{tab:nested_sH_pmed}Comparison of the different formulations on the instance set \pmed and setting \sH\ for $\mathcal{P} = \left\{p, p+1, p+2\right\}$}\\
\toprule
\multicolumn{3}{c}{ } & \multicolumn{4}{c}{\nPC} & \multicolumn{4}{c}{\nPCY} & \multicolumn{4}{c}{\nPCE} \\
\cmidrule(l){4-7} \cmidrule(l){8-11} \cmidrule(l){12-15}
idx & $|V|$ & p & t[s] & UB & g[\%] & bnb & t[s] & UB & g[\%] & bnb & t[s] & UB & g[\%] & bnb \\
\midrule
\endfirsthead
\toprule
\multicolumn{3}{c}{ } & \multicolumn{4}{c}{\nPC} & \multicolumn{4}{c}{\nPCY} & \multicolumn{4}{c}{\nPCE} \\
\cmidrule(l){4-7} \cmidrule(l){8-11} \cmidrule(l){12-15}
idx & $|V|$ & p & t[s] & UB & g[\%] & bnb & t[s] & UB & g[\%] & bnb & t[s] & UB & g[\%] & bnb \\
\midrule
\endhead
\midrule
\multicolumn{15}{r}{Continued on next page} \\
\midrule
\endfoot
\bottomrule
\endlastfoot
1 & 100 & 5 & 62.08 & 356 & 0.00 & 38 & 19.78 & 356 & 0.00 & 2501 & \textbf{5.58} & 356 & 0.00 & 48 \\
2 & 100 & 10 & 71.98 & 292 & 0.00 & 450 & 111.35 & 292 & 0.00 & 16882 & \textbf{4.36} & 292 & 0.00 & 954 \\
3 & 100 & 10 & 2.39 & 278 & 0.00 & 0 & 0.95 & 278 & 0.00 & 4 & \textbf{0.31} & 278 & 0.00 & 0 \\
4 & 100 & 20 & 0.51 & 220 & 0.00 & 0 & \textbf{0.23} & 220 & 0.00 & 0 & 0.25 & 220 & 0.00 & 0 \\
5 & 100 & 33 & 0.64 & 138 & 0.00 & 0 & 0.42 & 138 & 0.00 & 0 & \textbf{0.32} & 138 & 0.00 & 0 \\
6 & 200 & 5 & 2265.99 & 247 & 0.00 & 788 & 324.09 & 247 & 0.00 & 14400 & \textbf{19.10} & 247 & 0.00 & 97 \\
7 & 200 & 10 & 217.42 & 188 & 0.00 & 42 & 10.10 & 188 & 0.00 & 240 & \textbf{2.28} & 188 & 0.00 & 0 \\
8 & 200 & 20 & 208.62 & 161 & 0.00 & 131 & \textbf{1.71} & 161 & 0.00 & 3 & 2.55 & 161 & 0.00 & 0 \\
9 & 200 & 40 & 2.88 & 109 & 0.00 & 0 & 1.03 & 109 & 0.00 & 0 & \textbf{0.95} & 109 & 0.00 & 0 \\
10 & 200 & 67 & 1.99 & 58 & 0.00 & 0 & 0.92 & 58 & 0.00 & 0 & \textbf{0.90} & 58 & 0.00 & 0 \\
11 & 300 & 5 & 1790.01 & 170 & 0.00 & 12 & \textbf{2.85} & 170 & 0.00 & 17 & 10.93 & 170 & 0.00 & 0 \\
12 & 300 & 10 & 2158.37 & 151 & 0.00 & 220 & \textbf{4.56} & 151 & 0.00 & 62 & 6.82 & 151 & 0.00 & 9 \\
13 & 300 & 30 & 4.78 & 107 & 0.00 & 0 & \textbf{1.55} & 107 & 0.00 & 0 & 1.61 & 107 & 0.00 & 0 \\
14 & 300 & 60 & 10.06 & 76 & 0.00 & 0 & 2.81 & 76 & 0.00 & 0 & \textbf{2.05} & 76 & 0.00 & 0 \\
15 & 300 & 100 & 5.63 & 52 & 0.00 & 0 & 3.09 & 52 & 0.00 & 0 & \textbf{2.46} & 52 & 0.00 & 0 \\
16 & 400 & 5 & TL & 141 & 0.03 & 0 & \textbf{4.43} & 137 & 0.00 & 27 & 7.30 & 137 & 0.00 & 0 \\
17 & 400 & 10 & TL & 116 & 0.02 & 6 & \textbf{17.33} & 115 & 0.00 & 92 & 45.10 & 115 & 0.00 & 47 \\
18 & 400 & 40 & 98.57 & 83 & 0.00 & 0 & 4.49 & 83 & 0.00 & 0 & \textbf{4.13} & 83 & 0.00 & 0 \\
19 & 400 & 80 & 7.65 & 54 & 0.00 & 0 & \textbf{3.43} & 54 & 0.00 & 0 & 3.52 & 54 & 0.00 & 0 \\
20 & 400 & 133 & 8.21 & 39 & 0.00 & 0 & 4.03 & 39 & 0.00 & 0 & \textbf{3.91} & 39 & 0.00 & 0 \\
21 & 500 & 5 & TL & 118 & 0.03 & 0 & \textbf{5.28} & 116 & 0.00 & 7 & 11.86 & 116 & 0.00 & 0 \\
22 & 500 & 10 & TL & 114 & 0.03 & 0 & \textbf{395.23} & 113 & 0.00 & 2502 & 486.92 & 113 & 0.00 & 2222 \\
23 & 500 & 50 & 12.41 & 66 & 0.00 & 0 & \textbf{4.78} & 66 & 0.00 & 0 & 4.80 & 66 & 0.00 & 0 \\
24 & 500 & 100 & 12.04 & 45 & 0.00 & 0 & \textbf{5.95} & 45 & 0.00 & 0 & 5.98 & 45 & 0.00 & 0 \\
25 & 500 & 167 & 13.17 & 51 & 0.00 & 0 & 6.66 & 51 & 0.00 & 0 & \textbf{6.64} & 51 & 0.00 & 0 \\
26 & 600 & 5 & TL & 112 & 0.03 & 0 & \textbf{10.57} & 110 & 0.00 & 19 & 243.19 & 110 & 0.00 & 3 \\
27 & 600 & 10 & TL & 96 & 0.03 & 0 & \textbf{18.15} & 94 & 0.00 & 60 & 19.42 & 94 & 0.00 & 0 \\
28 & 600 & 60 & 18.25 & 54 & 0.00 & 0 & \textbf{7.89} & 54 & 0.00 & 0 & 7.95 & 54 & 0.00 & 0 \\
29 & 600 & 120 & 20.60 & 39 & 0.00 & 0 & \textbf{10.06} & 39 & 0.00 & 0 & 11.38 & 39 & 0.00 & 0 \\
30 & 600 & 200 & 20.39 & 45 & 0.00 & 0 & 11.16 & 45 & 0.00 & 0 & \textbf{10.63} & 45 & 0.00 & 0 \\
31 & 700 & 5 & 116.45 & 88 & 0.00 & 0 & \textbf{7.67} & 88 & 0.00 & 0 & 12.64 & 88 & 0.00 & 0 \\
32 & 700 & 10 & TL & 87 & 0.02 & 0 & \textbf{304.36} & 86 & 0.00 & 1134 & 1252.49 & 86 & 0.00 & 339 \\
33 & 700 & 70 & 23.83 & 45 & 0.00 & 0 & \textbf{12.02} & 45 & 0.00 & 0 & 12.05 & 45 & 0.00 & 0 \\
34 & 700 & 140 & 26.85 & 33 & 0.00 & 0 & 14.29 & 33 & 0.00 & 0 & \textbf{14.12} & 33 & 0.00 & 0 \\
35 & 800 & 5 & TL & 90 & 0.03 & 0 & \textbf{11.97} & 88 & 0.00 & 0 & 34.40 & 88 & 0.00 & 0 \\
36 & 800 & 10 & TL & 81 & 0.01 & 0 & \textbf{47.65} & 81 & 0.00 & 29 & 345.19 & 81 & 0.00 & 13 \\
37 & 800 & 80 & 33.39 & 45 & 0.00 & 0 & 15.74 & 45 & 0.00 & 0 & \textbf{15.08} & 45 & 0.00 & 0 \\
38 & 900 & 5 & TL & 86 & 0.05 & 0 & \textbf{21.45} & 84 & 0.00 & 19 & 87.42 & 84 & 0.00 & 38 \\
39 & 900 & 10 & TL & 69 & 0.01 & 0 & \textbf{25.46} & 69 & 0.00 & 5 & 405.36 & 69 & 0.00 & 9 \\
40 & 900 & 90 & 48.14 & 39 & 0.00 & 0 & \textbf{22.37} & 39 & 0.00 & 0 & 24.67 & 39 & 0.00 & 0 \\
\end{longtable}
\endgroup{}

%% file: media/tsplib_sH_detailed.tex
\begingroup\fontsize{6}{10}\selectfont
\begin{longtable}{lrrrrrrrrrrrr}
\caption{\label{tab:nested_sH_tsplib}Comparison of the different formulations on the instance set \tsplib and setting \sH\ for $\mathcal{P} = \left\{4, 5, 6\right\}$}\\
\toprule
  & \multicolumn{4}{c}{\nPC} & \multicolumn{4}{c}{\nPCY} & \multicolumn{4}{c}{\nPCE} \\
\cmidrule(l){2-5} \cmidrule(l){6-9} \cmidrule(l){10-13}
f & t[s] & UB & g[\%] & bnb & t[s] & UB & g[\%] & bnb & t[s] & UB & g[\%] & bnb \\
\midrule
\endfirsthead
\toprule
  & \multicolumn{4}{c}{\nPC} & \multicolumn{4}{c}{\nPCY} & \multicolumn{4}{c}{\nPCE} \\
\cmidrule(l){2-5} \cmidrule(l){6-9} \cmidrule(l){10-13}
f & t[s] & UB & g[\%] & bnb & t[s] & UB & g[\%] & bnb & t[s] & UB & g[\%] & bnb \\
\midrule
\endhead
\midrule
\multicolumn{13}{r}{Continued on next page} \\
\midrule
\endfoot
\bottomrule
\endlastfoot
eil51 & 6.94 & 61 & 0.00 & 89 & 1.99 & 61 & 0.00 & 448 & \textbf{0.84} & 61 & 0.00 & 53 \\
berlin52 & 0.78 & 1215 & 0.00 & 0 & \textbf{0.27} & 1215 & 0.00 & 9 & 0.33 & 1215 & 0.00 & 0 \\
st70 & 2.54 & 90 & 0.00 & 0 & 0.34 & 90 & 0.00 & 8 & \textbf{0.33} & 90 & 0.00 & 0 \\
eil76 & 21.92 & 64 & 0.00 & 100 & \textbf{0.95} & 64 & 0.00 & 53 & 1.36 & 64 & 0.00 & 28 \\
pr76 & 122.69 & 16330 & 0.00 & 749 & 38.63 & 16330 & 0.00 & 5563 & \textbf{16.65} & 16330 & 0.00 & 160 \\
rat99 & 471.08 & 144 & 0.00 & 1056 & 174.08 & 144 & 0.00 & 12230 & \textbf{7.56} & 144 & 0.00 & 121 \\
kroA100 & 40.37 & 2812 & 0.00 & 171 & 67.97 & 2812 & 0.00 & 5912 & \textbf{11.33} & 2812 & 0.00 & 232 \\
kroC100 & 520.80 & 2843 & 0.00 & 2973 & 862.44 & 2843 & 0.00 & 93601 & \textbf{39.07} & 2843 & 0.00 & 1071 \\
rd100 & 24.61 & 959 & 0.00 & 17 & \textbf{4.56} & 959 & 0.00 & 229 & 5.27 & 959 & 0.00 & 1 \\
kroB100 & 183.88 & 2866 & 0.00 & 600 & 193.94 & 2866 & 0.00 & 20626 & \textbf{13.18} & 2866 & 0.00 & 266 \\
kroD100 & 389.58 & 2862 & 0.00 & 1066 & 121.75 & 2862 & 0.00 & 11262 & \textbf{14.75} & 2862 & 0.00 & 120 \\
kroE100 & 1755.31 & 2893 & 0.00 & 7611 & 1174.39 & 2893 & 0.00 & 105046 & \textbf{152.64} & 2893 & 0.00 & 1034 \\
eil101 & 150.73 & 66 & 0.00 & 314 & 28.93 & 66 & 0.00 & 2056 & \textbf{2.75} & 66 & 0.00 & 53 \\
lin105 & 110.40 & 2067 & 0.00 & 432 & 38.89 & 2067 & 0.00 & 3101 & \textbf{9.88} & 2067 & 0.00 & 73 \\
pr107 & 16.23 & 5170 & 0.00 & 64 & 3398.78 & 5170 & 0.00 & 340384 & \textbf{0.99} & 5170 & 0.00 & 0 \\
pr124 & 10.18 & 7370 & 0.00 & 0 & 3.17 & 7370 & 0.00 & 119 & \textbf{1.59} & 7370 & 0.00 & 0 \\
bier127 & 94.81 & 15936 & 0.00 & 20 & \textbf{6.31} & 15936 & 0.00 & 326 & 8.47 & 15936 & 0.00 & 0 \\
ch130 & 417.33 & 664 & 0.00 & 459 & 67.85 & 664 & 0.00 & 3328 & \textbf{18.51} & 664 & 0.00 & 68 \\
pr136 & 42.51 & 9318 & 0.00 & 223 & 33.42 & 9318 & 0.00 & 2158 & \textbf{4.77} & 9318 & 0.00 & 37 \\
pr144 & 186.16 & 9853 & 0.00 & 155 & 172.28 & 9853 & 0.00 & 8706 & \textbf{8.80} & 9853 & 0.00 & 62 \\
ch150 & 557.21 & 647 & 0.00 & 608 & 163.22 & 647 & 0.00 & 8705 & \textbf{15.78} & 647 & 0.00 & 128 \\
kroB150 & TL & 2927 & 0.05 & 901 & 1069.37 & 2872 & 0.00 & 41919 & \textbf{163.30} & 2872 & 0.00 & 349 \\
kroA150 & TL & 3000 & 0.06 & 1114 & 2642.48 & 2934 & 0.00 & 110366 & \textbf{240.49} & 2934 & 0.00 & 1110 \\
pr152 & TL & 14430 & 0.07 & 1683 & TL & 14417 & 0.06 & 139711 & \textbf{408.30} & 14417 & 0.00 & 812 \\
u159 & 1348.38 & 4756 & 0.00 & 861 & 522.74 & 4756 & 0.00 & 26003 & \textbf{32.90} & 4756 & 0.00 & 278 \\
rat195 & TL & 210 & 0.08 & 162 & TL & 206 & 0.03 & 55800 & \textbf{397.21} & 205 & 0.00 & 582 \\
d198 & 1785.48 & 1583 & 0.00 & 23 & 248.14 & 1583 & 0.00 & 8066 & \textbf{55.42} & 1583 & 0.00 & 13 \\
kroA200 & TL & 3000 & 0.08 & 207 & TL & 2990 & 0.05 & 58393 & \textbf{1297.73} & 2976 & 0.00 & 1859 \\
kroB200 & TL & 2982 & 0.08 & 265 & TL & 2973 & 0.05 & 61558 & \textbf{899.48} & 2939 & 0.00 & 1074 \\
tsp225 & TL & 367 & 0.07 & 163 & TL & 367 & 0.04 & 61584 & \textbf{860.87} & 367 & 0.00 & 1885 \\
ts225 & 3029.51 & 12575 & 0.00 & 1171 & TL & 12575 & 0.04 & 47100 & \textbf{22.53} & 12575 & 0.00 & 313 \\
pr226 & 409.03 & 11812 & 0.00 & 214 & 129.50 & 11812 & 0.00 & 3403 & \textbf{16.97} & 11812 & 0.00 & 74 \\
gil262 & TL & 189 & 0.01 & 180 & 1252.54 & 189 & 0.00 & 21468 & \textbf{38.17} & 189 & 0.00 & 102 \\
pr264 & TL & 4809 & 0.03 & 538 & TL & 4809 & 0.04 & 52190 & \textbf{105.30} & 4809 & 0.00 & 1570 \\
a280 & TL & 237 & 0.13 & 6 & TL & 228 & 0.08 & 27558 & TL & 228 & 0.04 & 987 \\
pr299 & TL & 4322 & 0.06 & 101 & TL & 4312 & 0.05 & 25956 & \textbf{3403.87} & 4245 & 0.00 & 1373 \\
lin318 & TL & 3695 & 0.09 & 2 & TL & 3695 & 0.08 & 26439 & TL & 3686 & 0.08 & 321 \\
linhp318 & TL & 3695 & 0.09 & 0 & TL & 3695 & 0.08 & 25417 & TL & 3695 & 0.08 & 136 \\
rd400 & TL & 987 & 0.05 & 0 & TL & 963 & 0.01 & 19747 & \textbf{275.57} & 962 & 0.00 & 72 \\
fl417 & TL & 1714 & 0.04 & 0 & \textbf{19.90} & 1662 & 0.00 & 153 & 23.85 & 1662 & 0.00 & 0 \\
pr439 & TL & 10056 & 0.07 & 0 & \textbf{2115.38} & 9784 & 0.00 & 12554 & TL & 9789 & 0.02 & 78 \\
pcb442 & TL & 3321 & 0.07 & 3 & TL & 3321 & 0.06 & 12779 & TL & 3321 & 0.05 & 478 \\
d493 & TL & 2598 & 0.12 & 0 & TL & 2501 & 0.05 & 11138 & TL & 2598 & 0.10 & 2 \\
u574 & TL & 2594 & 0.14 & 0 & TL & 2585 & 0.11 & 10011 & TL & 2594 & 0.12 & 2 \\
rat575 & TL & 380 & 0.10 & 0 & TL & 380 & 0.09 & 7070 & TL & 380 & 0.09 & 12 \\
p654 & TL & 4935 & 0.12 & 0 & TL & 4790 & 0.06 & 9075 & TL & 4802 & 0.06 & 491 \\
d657 & TL & 3390 & 0.17 & 0 & TL & 3206 & 0.10 & 6769 & TL & 3390 & 0.14 & 0 \\
rat783 & TL & 435 & 0.09 & 0 & TL & 435 & 0.08 & 3619 & TL & 435 & 0.07 & 0 \\
pr1002 & TL & 12761 & 0.12 & 0 & TL & 12761 & 0.10 & 3096 & TL & 12761 & 0.09 & 0 \\
u1060 & TL & 11235 & 0.07 & 0 & TL & 11227 & 0.06 & 2811 & TL & 11235 & 0.06 & 0 \\
\end{longtable}
\endgroup{}

%% file: media/pmed_regret_detailed.tex
\begingroup\fontsize{6}{10}\selectfont
\begin{longtable}{lrrrrrrrrrrrrrr}
\caption{\label{tab:pmed_regret}Comparison of the different settings on the instance set \pmed for $\mathcal{P} = \left\{p, p+1, p+2\right\}$}\\
\toprule
\multicolumn{3}{c}{ } & \multicolumn{4}{c}{\R} & \multicolumn{4}{c}{\RN} & \multicolumn{4}{c}{\RNL} \\
\cmidrule(l){4-7} \cmidrule(l){8-11} \cmidrule(l){12-15}
idx & nNodes & p & t[s] & UB & g[\%] & bnb & t[s] & UB & g[\%] & bnb & t[s] & UB & g[\%] & bnb \\
\midrule
\endfirsthead
\toprule
\multicolumn{3}{c}{ } & \multicolumn{4}{c}{\R} & \multicolumn{4}{c}{\RN} & \multicolumn{4}{c}{\RNL} \\
\cmidrule(l){4-7} \cmidrule(l){8-11} \cmidrule(l){12-15}
idx & nNodes & p & t[s] & UB & g[\%] & bnb & t[s] & UB & g[\%] & bnb & t[s] & UB & g[\%] & bnb \\
\midrule
\endhead
\midrule
\multicolumn{15}{r}{Continued on next page} \\
\midrule
\endfoot
\bottomrule
\endlastfoot
1 & 100 & 5 & 33.44 & 0.03 & 0.00 & 4299 & 2.50 & 0.03 & 0.00 & 2 & \textbf{0.96} & 0.03 & 0.00 & 0 \\
2 & 100 & 10 & 240.19 & 0.04 & 0.00 & 29734 & 4.50 & 0.04 & 0.00 & 139 & \textbf{1.70} & 0.04 & 0.00 & 23 \\
3 & 100 & 10 & 66.25 & 0.01 & 0.00 & 6843 & 2.38 & 0.01 & 0.00 & 112 & \textbf{0.92} & 0.01 & 0.00 & 11 \\
4 & 100 & 20 & \textbf{0.48} & 0.00 & 0.00 & 0 & 0.49 & 0.00 & 0.00 & 0 & 0.53 & 0.00 & 0.00 & 0 \\
5 & 100 & 33 & 0.52 & 0.00 & 0.00 & 0 & \textbf{0.48} & 0.00 & 0.00 & 0 & 0.71 & 0.00 & 0.00 & 0 \\
6 & 200 & 5 & 663.87 & 0.03 & 0.00 & 30309 & \textbf{15.59} & 0.03 & 0.00 & 227 & 23.34 & 0.03 & 0.00 & 32 \\
7 & 200 & 10 & 216.79 & 0.05 & 0.00 & 10398 & 3.36 & 0.05 & 0.00 & 0 & \textbf{1.41} & 0.05 & 0.00 & 0 \\
8 & 200 & 20 & 3.96 & 0.02 & 0.00 & 54 & 2.68 & 0.02 & 0.00 & 0 & \textbf{1.48} & 0.02 & 0.00 & 0 \\
9 & 200 & 40 & 1.66 & 0.00 & 0.00 & 0 & 1.59 & 0.00 & 0.00 & 0 & \textbf{1.58} & 0.00 & 0.00 & 0 \\
10 & 200 & 67 & 1.72 & 0.00 & 0.00 & 0 & \textbf{1.67} & 0.00 & 0.00 & 0 & 1.69 & 0.00 & 0.00 & 0 \\
11 & 300 & 5 & 3.20 & 0.02 & 0.00 & 19 & 2.95 & 0.02 & 0.00 & 0 & \textbf{2.02} & 0.02 & 0.00 & 0 \\
12 & 300 & 10 & 9.73 & 0.02 & 0.00 & 189 & 13.66 & 0.02 & 0.00 & 91 & \textbf{2.41} & 0.02 & 0.00 & 0 \\
13 & 300 & 30 & 2.80 & 0.00 & 0.00 & 0 & 2.61 & 0.00 & 0.00 & 0 & \textbf{2.57} & 0.00 & 0.00 & 0 \\
14 & 300 & 60 & 4.37 & 0.00 & 0.00 & 0 & \textbf{4.36} & 0.00 & 0.00 & 0 & 4.42 & 0.00 & 0.00 & 0 \\
15 & 300 & 100 & 4.45 & 0.00 & 0.00 & 0 & \textbf{4.06} & 0.00 & 0.00 & 0 & 4.09 & 0.00 & 0.00 & 0 \\
16 & 400 & 5 & 28.04 & 0.02 & 0.00 & 415 & 5.30 & 0.02 & 0.00 & 3 & \textbf{3.04} & 0.02 & 0.00 & 7 \\
17 & 400 & 10 & 42.40 & 0.03 & 0.00 & 455 & 2023.61 & 0.03 & 0.00 & 17775 & \textbf{28.18} & 0.03 & 0.00 & 191 \\
18 & 400 & 40 & \textbf{5.75} & 0.00 & 0.00 & 0 & 6.10 & 0.00 & 0.00 & 0 & 5.87 & 0.00 & 0.00 & 0 \\
19 & 400 & 80 & \textbf{4.87} & 0.00 & 0.00 & 0 & 6.41 & 0.00 & 0.00 & 0 & 6.39 & 0.00 & 0.00 & 0 \\
20 & 400 & 133 & \textbf{5.33} & 0.00 & 0.00 & 0 & 6.45 & 0.00 & 0.00 & 0 & 6.56 & 0.00 & 0.00 & 0 \\
21 & 500 & 5 & 7.56 & 0.03 & 0.00 & 54 & 97.73 & 0.03 & 0.00 & 474 & \textbf{5.83} & 0.03 & 0.00 & 9 \\
22 & 500 & 10 & 1694.28 & 0.03 & 0.00 & 14421 & 2589.77 & 0.03 & 0.00 & 12489 & \textbf{84.71} & 0.03 & 0.00 & 320 \\
23 & 500 & 50 & \textbf{5.85} & 0.00 & 0.00 & 0 & 9.20 & 0.00 & 0.00 & 0 & 8.82 & 0.00 & 0.00 & 0 \\
24 & 500 & 100 & \textbf{7.18} & 0.00 & 0.00 & 0 & 8.77 & 0.00 & 0.00 & 0 & 8.77 & 0.00 & 0.00 & 0 \\
25 & 500 & 167 & \textbf{8.88} & 0.00 & 0.00 & 0 & 13.34 & 0.00 & 0.00 & 27 & 13.33 & 0.00 & 0.00 & 27 \\
26 & 600 & 5 & 15.80 & 0.03 & 0.00 & 29 & 271.23 & 0.03 & 0.00 & 1051 & \textbf{11.74} & 0.03 & 0.00 & 32 \\
27 & 600 & 10 & 128.86 & 0.03 & 0.00 & 2228 & TL & 0.03 & 5.88 & 12320 & \textbf{17.57} & 0.03 & 0.00 & 35 \\
28 & 600 & 60 & \textbf{8.76} & 0.00 & 0.00 & 0 & 10.99 & 0.00 & 0.00 & 0 & 10.91 & 0.00 & 0.00 & 0 \\
29 & 600 & 120 & \textbf{11.43} & 0.00 & 0.00 & 0 & 13.58 & 0.00 & 0.00 & 0 & 12.65 & 0.00 & 0.00 & 0 \\
30 & 600 & 200 & \textbf{14.98} & 0.00 & 0.00 & 0 & 18.09 & 0.00 & 0.00 & 0 & 18.98 & 0.00 & 0.00 & 0 \\
31 & 700 & 5 & \textbf{8.71} & 0.00 & 0.00 & 0 & 9.12 & 0.00 & 0.00 & 0 & 9.58 & 0.00 & 0.00 & 0 \\
32 & 700 & 10 & 211.53 & 0.03 & 0.00 & 360 & TL & 0.03 & 68.62 & 11051 & \textbf{132.47} & 0.03 & 0.00 & 374 \\
33 & 700 & 70 & \textbf{15.72} & 0.00 & 0.00 & 0 & 18.30 & 0.00 & 0.00 & 0 & 18.40 & 0.00 & 0.00 & 0 \\
34 & 700 & 140 & \textbf{17.05} & 0.00 & 0.00 & 0 & 19.84 & 0.00 & 0.00 & 0 & 19.98 & 0.00 & 0.00 & 0 \\
35 & 800 & 5 & \textbf{13.98} & 0.03 & 0.00 & 3 & 77.83 & 0.03 & 0.00 & 204 & 73.41 & 0.03 & 0.00 & 204 \\
36 & 800 & 10 & \textbf{64.02} & 0.04 & 0.00 & 116 & TL & 0.04 & 62.38 & 7925 & 149.47 & 0.04 & 0.00 & 289 \\
37 & 800 & 80 & \textbf{18.86} & 0.00 & 0.00 & 0 & 22.32 & 0.00 & 0.00 & 0 & 22.31 & 0.00 & 0.00 & 0 \\
38 & 900 & 5 & 77.04 & 0.04 & 0.00 & 276 & 203.46 & 0.04 & 0.00 & 329 & \textbf{20.01} & 0.04 & 0.00 & 0 \\
39 & 900 & 10 & 161.49 & 0.04 & 0.00 & 742 & TL & 0.05 & 47.22 & 6179 & \textbf{63.42} & 0.04 & 0.00 & 101 \\
40 & 900 & 90 & \textbf{24.52} & 0.00 & 0.00 & 0 & 24.67 & 0.00 & 0.00 & 0 & 24.68 & 0.00 & 0.00 & 0 \\
\end{longtable}
\endgroup{}

%% file: media/tsplib_regret_detailed.tex
\begingroup\fontsize{6}{10}\selectfont
\begin{longtable}{lrrrrrrrrrrrr}
\caption{\label{tab:tsplib_regret}Comparison of the different settings on the instance set \tsplib for $\mathcal{P} = \left\{4, 5, 6\right\}$}\\
\toprule
  & \multicolumn{4}{c}{\R} & \multicolumn{4}{c}{\RN} & \multicolumn{4}{c}{\RNL} \\
\cmidrule(l){2-5} \cmidrule(l){6-9} \cmidrule(l){10-13}
f & t[s] & UB & g[\%] & bnb & t[s] & UB & g[\%] & bnb & t[s] & UB & g[\%] & bnb \\
\midrule
\endfirsthead
\toprule
  & \multicolumn{4}{c}{\R} & \multicolumn{4}{c}{\RN} & \multicolumn{4}{c}{\RNL} \\
\cmidrule(l){2-5} \cmidrule(l){6-9} \cmidrule(l){10-13}
f & t[s] & UB & g[\%] & bnb & t[s] & UB & g[\%] & bnb & t[s] & UB & g[\%] & bnb \\
\midrule
\endhead
\midrule
\multicolumn{13}{r}{Continued on next page} \\
\midrule
\endfoot
\bottomrule
\endlastfoot
eil51 & 3.00 & 0.11 & 0.00 & 714 & 0.89 & 0.11 & 0.00 & 67 & \textbf{0.33} & 0.11 & 0.00 & 13 \\
berlin52 & 0.86 & 0.02 & 0.00 & 33 & 0.26 & 0.02 & 0.00 & 0 & \textbf{0.21} & 0.02 & 0.00 & 0 \\
st70 & 0.54 & 0.04 & 0.00 & 40 & 1.05 & 0.04 & 0.00 & 10 & \textbf{0.22} & 0.04 & 0.00 & 0 \\
eil76 & 1.64 & 0.09 & 0.00 & 159 & 3.00 & 0.09 & 0.00 & 3 & \textbf{0.56} & 0.09 & 0.00 & 7 \\
pr76 & 128.80 & 0.14 & 0.00 & 22353 & \textbf{0.68} & 0.14 & 0.00 & 0 & 0.75 & 0.14 & 0.00 & 0 \\
rat99 & 889.29 & 0.10 & 0.00 & 101249 & 2.89 & 0.10 & 0.00 & 0 & \textbf{1.30} & 0.10 & 0.00 & 0 \\
kroD100 & 297.36 & 0.11 & 0.00 & 29862 & 1.48 & 0.11 & 0.00 & 0 & \textbf{1.18} & 0.11 & 0.00 & 0 \\
kroA100 & 465.79 & 0.08 & 0.00 & 55858 & 4.00 & 0.08 & 0.00 & 0 & \textbf{3.19} & 0.08 & 0.00 & 0 \\
rd100 & 26.95 & 0.04 & 0.00 & 1994 & \textbf{0.52} & 0.04 & 0.00 & 0 & 0.53 & 0.04 & 0.00 & 0 \\
kroE100 & 1423.15 & 0.15 & 0.00 & 162374 & 15.57 & 0.15 & 0.00 & 230 & \textbf{11.53} & 0.15 & 0.00 & 165 \\
kroC100 & 993.51 & 0.14 & 0.00 & 112109 & 7.12 & 0.14 & 0.00 & 68 & \textbf{5.49} & 0.14 & 0.00 & 35 \\
kroB100 & 190.38 & 0.10 & 0.00 & 20177 & 2.97 & 0.10 & 0.00 & 0 & \textbf{2.19} & 0.10 & 0.00 & 0 \\
eil101 & 37.11 & 0.11 & 0.00 & 2697 & 2.60 & 0.11 & 0.00 & 0 & \textbf{0.50} & 0.11 & 0.00 & 0 \\
lin105 & 1476.80 & 0.11 & 0.00 & 160806 & 6.17 & 0.11 & 0.00 & 99 & \textbf{5.03} & 0.11 & 0.00 & 66 \\
pr107 & 928.37 & 0.39 & 0.00 & 100494 & 0.93 & 0.39 & 0.00 & 0 & \textbf{0.88} & 0.39 & 0.00 & 0 \\
pr124 & \textbf{0.49} & 0.03 & 0.00 & 0 & 0.54 & 0.03 & 0.00 & 0 & \textbf{0.49} & 0.03 & 0.00 & 0 \\
bier127 & 51.58 & 0.10 & 0.00 & 3647 & 2.28 & 0.10 & 0.00 & 34 & \textbf{1.86} & 0.10 & 0.00 & 18 \\
ch130 & 250.96 & 0.09 & 0.00 & 19738 & 4.22 & 0.09 & 0.00 & 32 & \textbf{2.02} & 0.09 & 0.00 & 4 \\
pr136 & 36.55 & 0.04 & 0.00 & 2454 & 2.38 & 0.04 & 0.00 & 0 & \textbf{2.18} & 0.04 & 0.00 & 0 \\
pr144 & 215.24 & 0.06 & 0.00 & 12996 & \textbf{4.01} & 0.06 & 0.00 & 47 & 5.62 & 0.06 & 0.00 & 61 \\
ch150 & 1290.34 & 0.07 & 0.00 & 77749 & 15.92 & 0.07 & 0.00 & 168 & \textbf{7.00} & 0.07 & 0.00 & 32 \\
kroB150 & TL & 0.10 & 7.80 & 178696 & \textbf{2.95} & 0.10 & 0.00 & 0 & 3.06 & 0.10 & 0.00 & 0 \\
kroA150 & TL & 0.10 & 35.54 & 138298 & 48.60 & 0.10 & 0.00 & 335 & \textbf{41.64} & 0.10 & 0.00 & 241 \\
pr152 & TL & 0.19 & 43.37 & 156245 & \textbf{0.91} & 0.19 & 0.00 & 0 & 1.24 & 0.19 & 0.00 & 0 \\
u159 & 804.38 & 0.12 & 0.00 & 43656 & 2.54 & 0.12 & 0.00 & 0 & \textbf{2.51} & 0.12 & 0.00 & 0 \\
rat195 & TL & 0.17 & 67.07 & 77032 & 36.94 & 0.17 & 0.00 & 269 & \textbf{13.21} & 0.17 & 0.00 & 43 \\
d198 & 361.00 & 0.22 & 0.00 & 13048 & 5.18 & 0.22 & 0.00 & 93 & \textbf{1.72} & 0.22 & 0.00 & 26 \\
kroB200 & TL & 0.15 & 73.35 & 75460 & \textbf{10.89} & 0.14 & 0.00 & 33 & 11.50 & 0.14 & 0.00 & 48 \\
kroA200 & TL & 0.15 & 75.06 & 74550 & \textbf{44.57} & 0.11 & 0.00 & 120 & 63.14 & 0.11 & 0.00 & 256 \\
ts225 & TL & 0.09 & 38.21 & 71546 & 51.04 & 0.09 & 0.00 & 85 & \textbf{7.24} & 0.09 & 0.00 & 2 \\
tsp225 & TL & 0.15 & 63.94 & 56608 & 3.40 & 0.13 & 0.00 & 0 & \textbf{2.22} & 0.13 & 0.00 & 0 \\
pr226 & 851.04 & 0.06 & 0.00 & 31617 & 12.23 & 0.06 & 0.00 & 56 & \textbf{10.87} & 0.06 & 0.00 & 44 \\
gil262 & 2307.01 & 0.05 & 0.00 & 33471 & 2.81 & 0.05 & 0.00 & 0 & \textbf{1.68} & 0.05 & 0.00 & 0 \\
pr264 & 2529.76 & 0.21 & 0.00 & 67087 & 1.55 & 0.21 & 0.00 & 0 & \textbf{1.27} & 0.21 & 0.00 & 0 \\
a280 & TL & 0.21 & 83.33 & 34970 & 696.16 & 0.16 & 0.00 & 1857 & \textbf{423.27} & 0.16 & 0.00 & 1016 \\
pr299 & TL & 0.10 & 87.01 & 34766 & \textbf{234.22} & 0.07 & 0.00 & 609 & 250.97 & 0.07 & 0.00 & 366 \\
lin318 & TL & 0.14 & 82.86 & 31958 & 25.76 & 0.11 & 0.00 & 80 & \textbf{22.37} & 0.11 & 0.00 & 21 \\
linhp318 & TL & 0.14 & 82.86 & 31940 & 25.87 & 0.11 & 0.00 & 80 & \textbf{21.09} & 0.11 & 0.00 & 21 \\
rd400 & TL & 0.06 & 64.71 & 19414 & 60.48 & 0.06 & 0.00 & 105 & \textbf{29.18} & 0.06 & 0.00 & 46 \\
fl417 & 446.30 & 0.04 & 0.00 & 5195 & \textbf{1.79} & 0.04 & 0.00 & 0 & 1.88 & 0.04 & 0.00 & 0 \\
pr439 & TL & 0.14 & 80.33 & 18388 & 89.70 & 0.12 & 0.00 & 231 & \textbf{54.96} & 0.12 & 0.00 & 127 \\
pcb442 & TL & 0.17 & 87.16 & 17840 & \textbf{826.11} & 0.11 & 0.00 & 959 & 1456.50 & 0.11 & 0.00 & 1120 \\
d493 & TL & 0.14 & 66.97 & 13677 & \textbf{31.17} & 0.13 & 0.00 & 33 & 62.35 & 0.13 & 0.00 & 57 \\
u574 & TL & 0.24 & 88.97 & 10790 & 22.00 & 0.16 & 0.00 & 0 & \textbf{11.01} & 0.16 & 0.00 & 0 \\
rat575 & TL & 0.21 & 90.03 & 8069 & 1514.22 & 0.14 & 0.00 & 1285 & \textbf{245.72} & 0.14 & 0.00 & 181 \\
p654 & TL & 0.12 & 41.69 & 10939 & 16.01 & 0.12 & 0.00 & 19 & \textbf{12.32} & 0.12 & 0.00 & 15 \\
d657 & TL & 0.22 & 89.06 & 8460 & 109.94 & 0.15 & 0.00 & 60 & \textbf{59.95} & 0.15 & 0.00 & 23 \\
rat783 & TL & 0.27 & 92.56 & 5473 & 3036.54 & 0.14 & 0.00 & 1089 & \textbf{622.85} & 0.14 & 0.00 & 194 \\
pr1002 & TL & 0.29 & 92.79 & 3650 & TL & 0.17 & 33.99 & 410 & TL & 0.18 & 44.75 & 278 \\
u1060 & TL & 0.11 & 91.60 & 3499 & 901.77 & 0.07 & 0.00 & 449 & \textbf{666.29} & 0.07 & 0.00 & 315 \\
\end{longtable}
\endgroup{}

%% file: nestedpCenterArXiV.bbl
\begin{thebibliography}{41}
\expandafter\ifx\csname natexlab\endcsname\relax\def\natexlab#1{#1}\fi
\providecommand{\url}[1]{\texttt{#1}}
\providecommand{\href}[2]{#2}
\providecommand{\path}[1]{#1}
\providecommand{\DOIprefix}{doi:}
\providecommand{\ArXivprefix}{arXiv:}
\providecommand{\URLprefix}{URL: }
\providecommand{\Pubmedprefix}{pmid:}
\providecommand{\doi}[1]{\href{http://dx.doi.org/#1}{\path{#1}}}
\providecommand{\Pubmed}[1]{\href{pmid:#1}{\path{#1}}}
\providecommand{\bibinfo}[2]{#2}
\ifx\xfnm\relax \def\xfnm[#1]{\unskip,\space#1}\fi
%Type = Article
\bibitem[{Albareda-Sambola et~al.(2009)Albareda-Sambola, Fern{\'a}ndez, Hinojosa and Puerto}]{albareda2009}
\bibinfo{author}{Albareda-Sambola, M.}, \bibinfo{author}{Fern{\'a}ndez, E.}, \bibinfo{author}{Hinojosa, Y.}, \bibinfo{author}{Puerto, J.}, \bibinfo{year}{2009}.
\newblock \bibinfo{title}{The multi-period incremental service facility location problem}.
\newblock \bibinfo{journal}{Computers \& Operations Research} \bibinfo{volume}{36}, \bibinfo{pages}{1356--1375}.
%Type = Article
\bibitem[{Albareda-Sambola et~al.(2015)Albareda-Sambola, Hinojosa, Mar{\'\i}n and Puerto}]{albareda2015centers}
\bibinfo{author}{Albareda-Sambola, M.}, \bibinfo{author}{Hinojosa, Y.}, \bibinfo{author}{Mar{\'\i}n, A.}, \bibinfo{author}{Puerto, J.}, \bibinfo{year}{2015}.
\newblock \bibinfo{title}{When centers can fail: A close second opportunity}.
\newblock \bibinfo{journal}{Computers \& Operations Research} \bibinfo{volume}{62}, \bibinfo{pages}{145--156}.
%Type = Inproceedings
\bibitem[{Ales and Elloumi(2018)}]{Elloumi2018}
\bibinfo{author}{Ales, Z.}, \bibinfo{author}{Elloumi, S.}, \bibinfo{year}{2018}.
\newblock \bibinfo{title}{Compact {MILP} formulations for the p-center problem}, in: \bibinfo{booktitle}{Combinatorial Optimization: 5th International Symposium, ISCO 2018}, \bibinfo{publisher}{Springer International Publishing}. pp. \bibinfo{pages}{14--25}.
%Type = Article
\bibitem[{Arulselvan et~al.(2019)Arulselvan, Bley and Ljubić}]{Ljubic2019}
\bibinfo{author}{Arulselvan, A.}, \bibinfo{author}{Bley, A.}, \bibinfo{author}{Ljubić, I.}, \bibinfo{year}{2019}.
\newblock \bibinfo{title}{The incremental connected facility location problem}.
\newblock \bibinfo{journal}{Computers \& Operations Research} \bibinfo{volume}{112}, \bibinfo{pages}{104763}.
%Type = Article
\bibitem[{Bakker and Nickel(2024)}]{bakker2024value}
\bibinfo{author}{Bakker, H.}, \bibinfo{author}{Nickel, S.}, \bibinfo{year}{2024}.
\newblock \bibinfo{title}{The value of the multi-period solution revisited: When to model time in capacitated location problems}.
\newblock \bibinfo{journal}{Computers \& Operations Research} \bibinfo{volume}{161}, \bibinfo{pages}{106428}.
%Type = Misc
\bibitem[{Brandstetter(2024)}]{Brandstetter2024b}
\bibinfo{author}{Brandstetter, C.}, \bibinfo{year}{2024}.
\newblock \bibinfo{title}{On the nested p-center problem}.
\newblock \bibinfo{note}{Master's thesis, Johannes Kepler University Linz, 2024, available at \url{https://epub.jku.at/obvulihs/content/titleinfo/9842417}}.
%Type = Article
\bibitem[{Brandstetter and Sinnl(2024)}]{Brandstetter2024}
\bibinfo{author}{Brandstetter, C.}, \bibinfo{author}{Sinnl, M.}, \bibinfo{year}{2024}.
\newblock \bibinfo{title}{On the nested p-center problem}.
\newblock \bibinfo{journal}{International Network and Optimization Conference 2024 Proceedings} .
%Type = Article
\bibitem[{Calik and Tansel(2013)}]{calik2013double}
\bibinfo{author}{Calik, H.}, \bibinfo{author}{Tansel, B.C.}, \bibinfo{year}{2013}.
\newblock \bibinfo{title}{Double bound method for solving the p-center location problem}.
\newblock \bibinfo{journal}{Computers \& Operations Research} \bibinfo{volume}{40}, \bibinfo{pages}{2991--2999}.
%Type = Article
\bibitem[{Castro et~al.(2017)Castro, Nasini and Saldanha-da Gama}]{castro2017cutting}
\bibinfo{author}{Castro, J.}, \bibinfo{author}{Nasini, S.}, \bibinfo{author}{Saldanha-da Gama, F.}, \bibinfo{year}{2017}.
\newblock \bibinfo{title}{A cutting-plane approach for large-scale capacitated multi-period facility location using a specialized interior-point method}.
\newblock \bibinfo{journal}{Mathematical Programming} \bibinfo{volume}{163}, \bibinfo{pages}{411--444}.
%Type = Article
\bibitem[{Chen and Chen(2009)}]{chen2009new}
\bibinfo{author}{Chen, D.}, \bibinfo{author}{Chen, R.}, \bibinfo{year}{2009}.
\newblock \bibinfo{title}{New relaxation-based algorithms for the optimal solution of the continuous and discrete p-center problems}.
\newblock \bibinfo{journal}{Computers \& Operations Research} \bibinfo{volume}{36}, \bibinfo{pages}{1646--1655}.
%Type = Article
\bibitem[{Contardo et~al.(2019)Contardo, Iori and Kramer}]{contardo2019scalable}
\bibinfo{author}{Contardo, C.}, \bibinfo{author}{Iori, M.}, \bibinfo{author}{Kramer, R.}, \bibinfo{year}{2019}.
\newblock \bibinfo{title}{A scalable exact algorithm for the vertex p-center problem}.
\newblock \bibinfo{journal}{Computers \& Operations Research} \bibinfo{volume}{103}, \bibinfo{pages}{211--220}.
%Type = Article
\bibitem[{Correia and Melo(2016)}]{CORREIA2016}
\bibinfo{author}{Correia, I.}, \bibinfo{author}{Melo, T.}, \bibinfo{year}{2016}.
\newblock \bibinfo{title}{Multi-period capacitated facility location under delayed demand satisfaction}.
\newblock \bibinfo{journal}{European Journal of Operational Research} \bibinfo{volume}{255}, \bibinfo{pages}{729--746}.
%Type = Article
\bibitem[{Correia and Melo(2017)}]{Correia2017}
\bibinfo{author}{Correia, I.}, \bibinfo{author}{Melo, T.}, \bibinfo{year}{2017}.
\newblock \bibinfo{title}{A multi-period facility location problem with modular capacity adjustments and flexible demand fulfillment}.
\newblock \bibinfo{journal}{Computers \& Industrial Engineering} \bibinfo{volume}{110}, \bibinfo{pages}{307--321}.
%Type = Book
\bibitem[{Daskin(2013)}]{wiley2013}
\bibinfo{author}{Daskin, M.S.}, \bibinfo{year}{2013}.
\newblock \bibinfo{title}{Network and Discrete Location: Models, Algorithms, and Applications, Second Edition}.
\newblock \bibinfo{publisher}{John Wiley \& Sons, Ltd}.
%Type = Article
\bibitem[{Elloumi et~al.(2004)Elloumi, Labb{\'e} and Pochet}]{Elloumi2004}
\bibinfo{author}{Elloumi, S.}, \bibinfo{author}{Labb{\'e}, M.}, \bibinfo{author}{Pochet, Y.}, \bibinfo{year}{2004}.
\newblock \bibinfo{title}{A new formulation and resolution method for the p-center problem}.
\newblock \bibinfo{journal}{INFORMS Journal on Computing} \bibinfo{volume}{16}, \bibinfo{pages}{84--94}.
%Type = Article
\bibitem[{Escudero and Pizarro~Romero(2017)}]{escudero2017solving}
\bibinfo{author}{Escudero, L.F.}, \bibinfo{author}{Pizarro~Romero, C.}, \bibinfo{year}{2017}.
\newblock \bibinfo{title}{On solving a large-scale problem on facility location and customer assignment with interaction costs along a time horizon}.
\newblock \bibinfo{journal}{Top} \bibinfo{volume}{25}, \bibinfo{pages}{601--622}.
%Type = Article
\bibitem[{Fischetti et~al.(2017)Fischetti, Ljubi\'{c} and Sinnl}]{Fischetti2017}
\bibinfo{author}{Fischetti, M.}, \bibinfo{author}{Ljubi\'{c}, I.}, \bibinfo{author}{Sinnl, M.}, \bibinfo{year}{2017}.
\newblock \bibinfo{title}{Redesigning benders decomposition for large-scale facility location}.
\newblock \bibinfo{journal}{Management Science} \bibinfo{volume}{63}, \bibinfo{pages}{2146--2162}.
%Type = Article
\bibitem[{Friedler and Mount(2010)}]{friedler2010approximation}
\bibinfo{author}{Friedler, S.A.}, \bibinfo{author}{Mount, D.M.}, \bibinfo{year}{2010}.
\newblock \bibinfo{title}{Approximation algorithm for the kinetic robust k-center problem}.
\newblock \bibinfo{journal}{Computational Geometry} \bibinfo{volume}{43}, \bibinfo{pages}{572--586}.
%Type = Article
\bibitem[{Gaar and Sinnl(2022)}]{GAAR2022}
\bibinfo{author}{Gaar, E.}, \bibinfo{author}{Sinnl, M.}, \bibinfo{year}{2022}.
\newblock \bibinfo{title}{A scaleable projection-based branch-and-cut algorithm for the p-center problem}.
\newblock \bibinfo{journal}{European Journal of Operational Research} \bibinfo{volume}{303}, \bibinfo{pages}{78--98}.
%Type = Article
\bibitem[{Gaar and Sinnl(2023)}]{gaar2023exact}
\bibinfo{author}{Gaar, E.}, \bibinfo{author}{Sinnl, M.}, \bibinfo{year}{2023}.
\newblock \bibinfo{title}{Exact solution approaches for the discrete $\alpha$-neighbor p-center problem}.
\newblock \bibinfo{journal}{Networks} \bibinfo{volume}{82}, \bibinfo{pages}{371--399}.
%Type = Article
\bibitem[{Galv{\~a}o and Santiba{\~n}ez-Gonzalez(1992)}]{galvao1992}
\bibinfo{author}{Galv{\~a}o, R.D.}, \bibinfo{author}{Santiba{\~n}ez-Gonzalez, E.d.R.}, \bibinfo{year}{1992}.
\newblock \bibinfo{title}{{A Lagrangean heuristic for the pk-median dynamic location problem}}.
\newblock \bibinfo{journal}{European Journal of Operational Research} \bibinfo{volume}{58}, \bibinfo{pages}{250--262}.
%Type = Article
\bibitem[{Garcia-Diaz et~al.(2019)Garcia-Diaz, Menchaca-Mendez, Menchaca-Mendez, Hern{\'a}ndez, P{\'e}rez-Sansalvador and Lakouari}]{garcia-diaz2019}
\bibinfo{author}{Garcia-Diaz, J.}, \bibinfo{author}{Menchaca-Mendez, R.}, \bibinfo{author}{Menchaca-Mendez, R.}, \bibinfo{author}{Hern{\'a}ndez, S.P.}, \bibinfo{author}{P{\'e}rez-Sansalvador, J.C.}, \bibinfo{author}{Lakouari, N.}, \bibinfo{year}{2019}.
\newblock \bibinfo{title}{Approximation algorithms for the vertex k-center problem: Survey and experimental evaluation}.
\newblock \bibinfo{journal}{IEEE Access} \bibinfo{volume}{7}, \bibinfo{pages}{109228--109245}.
%Type = Article
\bibitem[{Güden and Süral(2019)}]{GUDEN2019}
\bibinfo{author}{Güden, H.}, \bibinfo{author}{Süral, H.}, \bibinfo{year}{2019}.
\newblock \bibinfo{title}{The dynamic p-median problem with mobile facilities}.
\newblock \bibinfo{journal}{Computers \& Industrial Engineering} \bibinfo{volume}{135}, \bibinfo{pages}{615--627}.
%Type = Article
\bibitem[{Hakimi(1964)}]{Hakimi1964}
\bibinfo{author}{Hakimi, S.L.}, \bibinfo{year}{1964}.
\newblock \bibinfo{title}{Optimum locations of switching centers and the absolute centers and medians of a graph}.
\newblock \bibinfo{journal}{Operations Research} \bibinfo{volume}{12}, \bibinfo{pages}{450--459}.
%Type = Article
\bibitem[{Jia et~al.(2007)Jia, Ord{\'o}{\~n}ez and Dessouky}]{jia2007modeling}
\bibinfo{author}{Jia, H.}, \bibinfo{author}{Ord{\'o}{\~n}ez, F.}, \bibinfo{author}{Dessouky, M.}, \bibinfo{year}{2007}.
\newblock \bibinfo{title}{A modeling framework for facility location of medical services for large-scale emergencies}.
\newblock \bibinfo{journal}{IIE Transactions} \bibinfo{volume}{39}, \bibinfo{pages}{41--55}.
%Type = Article
\bibitem[{Kalinowski et~al.(2015)Kalinowski, Matsypura and Savelsbergh}]{Kalinowski2015}
\bibinfo{author}{Kalinowski, T.}, \bibinfo{author}{Matsypura, D.}, \bibinfo{author}{Savelsbergh, M.W.}, \bibinfo{year}{2015}.
\newblock \bibinfo{title}{Incremental network design with maximum flows}.
\newblock \bibinfo{journal}{European Journal of Operational Research} \bibinfo{volume}{242}, \bibinfo{pages}{51--62}.
%Type = Article
\bibitem[{Kariv and Hakimi(1979)}]{kariv1979}
\bibinfo{author}{Kariv, O.}, \bibinfo{author}{Hakimi, S.L.}, \bibinfo{year}{1979}.
\newblock \bibinfo{title}{An algorithmic approach to network location problems. i: The p-centers}.
\newblock \bibinfo{journal}{SIAM Journal on Applied Mathematics} \bibinfo{volume}{37}, \bibinfo{pages}{513--538}.
%Type = Article
\bibitem[{Khuller and Sussmann(2000)}]{khuller2000capacitated}
\bibinfo{author}{Khuller, S.}, \bibinfo{author}{Sussmann, Y.J.}, \bibinfo{year}{2000}.
\newblock \bibinfo{title}{The capacitated k-center problem}.
\newblock \bibinfo{journal}{SIAM Journal on Discrete Mathematics} \bibinfo{volume}{13}, \bibinfo{pages}{403--418}.
%Type = Book
\bibitem[{Laporte et~al.(2019)Laporte, Nickel and Saldanha-da Gama}]{Laporte2019}
\bibinfo{author}{Laporte, G.}, \bibinfo{author}{Nickel, S.}, \bibinfo{author}{Saldanha-da Gama, F.}, \bibinfo{year}{2019}.
\newblock \bibinfo{title}{Location Science}.
\newblock \bibinfo{publisher}{Springer}.
%Type = Article
\bibitem[{Lu and Sheu(2013)}]{lu2013robust}
\bibinfo{author}{Lu, C.C.}, \bibinfo{author}{Sheu, J.B.}, \bibinfo{year}{2013}.
\newblock \bibinfo{title}{Robust vertex {$p$}-center model for locating urgent relief distribution centers}.
\newblock \bibinfo{journal}{Computers \& Operations Research} \bibinfo{volume}{40}, \bibinfo{pages}{2128--2137}.
%Type = Article
\bibitem[{Malkomes et~al.(2015)Malkomes, Kusner, Chen, Weinberger and Moseley}]{malkomes2015fast}
\bibinfo{author}{Malkomes, G.}, \bibinfo{author}{Kusner, M.J.}, \bibinfo{author}{Chen, W.}, \bibinfo{author}{Weinberger, K.Q.}, \bibinfo{author}{Moseley, B.}, \bibinfo{year}{2015}.
\newblock \bibinfo{title}{Fast distributed {$k$}-center clustering with outliers on massive data}.
\newblock \bibinfo{journal}{Advances in Neural Information Processing Systems} \bibinfo{volume}{28}, \bibinfo{pages}{1063--1071}.
%Type = Article
\bibitem[{McGarvey and Thorsen(2022)}]{McGarvey2022}
\bibinfo{author}{McGarvey, R.G.}, \bibinfo{author}{Thorsen, A.}, \bibinfo{year}{2022}.
\newblock \bibinfo{title}{Nested-solution facility location models}.
\newblock \bibinfo{journal}{Optimization Letters} \bibinfo{volume}{16}, \bibinfo{pages}{497--514}.
%Type = Article
\bibitem[{Meinl et~al.(2011)Meinl, Ostermann and Berthold}]{meinl2011maximum}
\bibinfo{author}{Meinl, T.}, \bibinfo{author}{Ostermann, C.}, \bibinfo{author}{Berthold, M.R.}, \bibinfo{year}{2011}.
\newblock \bibinfo{title}{Maximum-score diversity selection for early drug discovery}.
\newblock \bibinfo{journal}{Journal of Chemical Information and Modeling} \bibinfo{volume}{51}, \bibinfo{pages}{237--247}.
%Type = Article
\bibitem[{Minieka(1970)}]{minieka1970}
\bibinfo{author}{Minieka, E.}, \bibinfo{year}{1970}.
\newblock \bibinfo{title}{The m-center problem}.
\newblock \bibinfo{journal}{SIAM Review} \bibinfo{volume}{12}, \bibinfo{pages}{138--139}.
%Type = Article
\bibitem[{Reinelt(1991)}]{Reinelt1991}
\bibinfo{author}{Reinelt, G.}, \bibinfo{year}{1991}.
\newblock \bibinfo{title}{{TSPLIB}—a traveling salesman problem library}.
\newblock \bibinfo{journal}{ORSA Journal on Computing} \bibinfo{volume}{3}, \bibinfo{pages}{376--384}.
%Type = Article
\bibitem[{Roodman and Schwarz(1975)}]{Roodman1975}
\bibinfo{author}{Roodman, G.M.}, \bibinfo{author}{Schwarz, L.B.}, \bibinfo{year}{1975}.
\newblock \bibinfo{title}{Optimal and heuristic facility phase-out strategies}.
\newblock \bibinfo{journal}{AIIE Transactions} \bibinfo{volume}{7}, \bibinfo{pages}{177--184}.
%Type = Article
\bibitem[{Sauvey et~al.(2020)Sauvey, Melo and Correia}]{SAUVEY2020}
\bibinfo{author}{Sauvey, C.}, \bibinfo{author}{Melo, T.}, \bibinfo{author}{Correia, I.}, \bibinfo{year}{2020}.
\newblock \bibinfo{title}{Heuristics for a multi-period facility location problem with delayed demand satisfaction}.
\newblock \bibinfo{journal}{Computers \& Industrial Engineering} \bibinfo{volume}{139}, \bibinfo{pages}{106171}.
%Type = Article
\bibitem[{Schilling(1980)}]{Schilling1980}
\bibinfo{author}{Schilling, D.A.}, \bibinfo{year}{1980}.
\newblock \bibinfo{title}{Dynamic location modeling for public-sector facilities: A multicriteria approach}.
\newblock \bibinfo{journal}{Decision Sciences} \bibinfo{volume}{11}, \bibinfo{pages}{714--724}.
%Type = Article
\bibitem[{Scott(1971)}]{Scott1971}
\bibinfo{author}{Scott, A.J.}, \bibinfo{year}{1971}.
\newblock \bibinfo{title}{Dynamic location-allocation systems: some basic planning strategies}.
\newblock \bibinfo{journal}{Environment and Planning A} \bibinfo{volume}{3}, \bibinfo{pages}{73--82}.
%Type = Article
\bibitem[{{van den Berg} and Aardal(2015)}]{VANDENBERG2015}
\bibinfo{author}{{van den Berg}, P.L.}, \bibinfo{author}{Aardal, K.}, \bibinfo{year}{2015}.
\newblock \bibinfo{title}{Time-dependent mexclp with start-up and relocation cost}.
\newblock \bibinfo{journal}{European Journal of Operational Research} \bibinfo{volume}{242}, \bibinfo{pages}{383--389}.
%Type = Article
\bibitem[{Wesolowsky and Truscott(1975)}]{Wesolowsky1975}
\bibinfo{author}{Wesolowsky, G.O.}, \bibinfo{author}{Truscott, W.G.}, \bibinfo{year}{1975}.
\newblock \bibinfo{title}{The multiperiod location-allocation problem with relocation of facilities}.
\newblock \bibinfo{journal}{Management Science} \bibinfo{volume}{22}, \bibinfo{pages}{57--65}.

\end{thebibliography}
